\documentclass[aos]{imsart}

\RequirePackage[OT1]{fontenc}
\RequirePackage{amsthm, amsmath, amsfonts, amssymb}
\RequirePackage[square]{natbib}
\RequirePackage{hypernat}

\RequirePackage{graphicx, color}

\arxiv{math.PR/0000000}

\usepackage[psamsfonts]{amssymb}
\usepackage{verbatim,enumerate,xspace}
\usepackage[mathscr]{eucal}

\usepackage{hyperref, url}
\usepackage{subfigure}

\usepackage{bbm}
\usepackage{dsfont}
\usepackage{epic, eepic}

\usepackage{breakurl}

\newcommand{\bsl}{\texttt{\symbol{92}}}                  

\def\E{\mathrm{E}}

\def\bd{\begin{description}}
\def\ed{\end{description}}
\def\bl{\begin{list}{$\bullet$}{}}   
\def\cl{\begin{list}{$\circ$}{}}     
\def\el{\end{list}}                  

\def\b1{\mathbf{1}}

\newcommand{\PP}{{\mathbb P}}

\newcommand{\RR}{\mathds{R}}

\newcommand{\YY}{{\mathbb Y}}
\newcommand{\HH}{{\mathbb H}}
\newcommand{\NN}{\mathds{N}}
\newcommand{\ZZ}{{\mathbb Z}}

\newcommand{\ud}{\mathrm{d}}

\newcommand{\ff}{\ensuremath{\mathcal {F}}}

\newcommand{\ffb}{\ensuremath{\ff_{\text{BDD}}}}
\newcommand{\ffs}{\ensuremath{\ff_{\text{SMU}}}}

\renewcommand{\qedsymbol}{\ensuremath{\blacksquare}}

\newcommand{\s}{\ensuremath{(0,\infty)^d}}           
\newcommand{\m}[1]{\ensuremath{\boldsymbol{#1}}}     

\newcommand{\lo}{\ensuremath{o}}
\newcommand{\bo}{\ensuremath{\mathcal{O}}}

\newcommand{\xx}{\ensuremath{x_{0i}-h_i n^{-\frac{1}{3d}}}}
\newcommand{\XX}{\ensuremath{x_{0i}+h_i n^{-\frac{1}{3d}}}}

\newenvironment{enumeratei}{\begin{enumerate}[\upshape (i)]}%
                           {\end{enumerate}}
\newenvironment{enumeratea}{\begin{enumerate}[\upshape (a)]}%
                           {\end{enumerate}}

\theoremstyle{plain}
\newtheorem{theorem}{Theorem}[section]

\newtheorem{lemma}[theorem]{Lemma}
\newtheorem{proposition}[theorem]{Proposition}
\newtheorem{assumption}{Assumption}[section]

\theoremstyle{definition}
\newtheorem{definition}{Definition}[section]
\newtheorem{example}{Example}[section]

\theoremstyle{remark}
\newtheorem*{remark}{Remark}

\numberwithin{equation}{section}

\begin{document}

\begin{frontmatter}

\title{Nonparametric estimation of multivariate scale mixtures of uniform densities}
\runtitle{Multivariate monotone densities}
\begin{aug}
    \author{\fnms{Marios~G.} \snm{Pavlides}\thanksref{t1,m1}\ead[label=e1]{ m.pavlides@frederick.ac.cy}} 
	\and 
    \author{\fnms{Jon A.} \snm{Wellner}\thanksref{t2,m2}\ead[label=e2]{jaw@stat.washington.edu}}
	\thankstext{t1}{Research supported in part by NSF grant DMS-0503822}
	\thankstext{t2}{Research supported in part by NSF grant DMS-0804587 and NI-AID grant 2R01 AI291968-04}     
	\runauthor{Marios G. Pavlides and Jon A. Wellner}
        \affiliation{Frederick University, Nicosio, Cyprus\thanksmark{m1}}
	\affiliation{University of Washington\thanksmark{m2}}

	\address{Frederick University\\Nicosia 1036\\ Cyprus\\
	\printead{e1}}
	\address{Department of Statistics, Box 354322\\University of Washington\\Seattle, WA  98195-4322\\
	\printead{e2}}     
\end{aug}

\begin{abstract}
        Suppose that $\m{U} = (U_1, \ldots , U_d) $  has a Uniform$([0,1]^d)$ distribution,
        that $\m{Y} = (Y_1 , \ldots , Y_d) $ has the distribution $G$ on $\RR_+^d$, and
        let $\m{X} = (X_1 , \ldots , X_d) = (U_1 Y_1 , \ldots , U_d Y_d )$.
        The resulting class of distributions of $\m{X}$ (as $G$ varies over all distributions on
        $\RR_+^d$) is called the {\sl Scale Mixture of Uniforms} class of distributions, and the corresponding class
        of densities on $\RR_+^d$ is denoted by $\ffs(d)$.
        We study  maximum likelihood estimation in the family ${\ffs}(d)$.
        We prove existence of the MLE, establish Fenchel
        characterizations, and prove strong consistency of the almost
        surely unique maximum likelihood estimator (MLE) in $\ffs(d)$.
        We also provide an asymptotic
        minimax lower bound for estimating the functional $f \mapsto
        f(\m{x})$ under reasonable differentiability assumptions on $f\in\ffs(d)$ in a
        neighborhood of  $\m{x}$. We conclude the paper with discussion, conjectures and open
        problems pertaining to global and local rates of convergence of the MLE.
\end{abstract}

\begin{keyword}[class=AMS]
\kwd[Primary ]{62G07}
\kwd{62H12}
\kwd[; secondary ]{62G05}
\kwd{62G20}
\kwd{62F20}
\kwd{62H12}
\end{keyword}

\begin{keyword}
\kwd{consistency}
\kwd{minimax}
\kwd{monotonicity}
\kwd{multivariate}
\kwd{maximum likelihood}
\kwd{uniform}
\kwd{mixture}
\end{keyword}

\end{frontmatter}

\section{Introduction and summary}\label{S: intro}

        Fix a non-negative integer $k$, and suppose that $X_1,\dots,X_n$ are
        i.i.d. random variables distributed according to a density in the
        convex family of {\sl $k$-monotone}
        densities (with respect to Lebesgue measure) on $(0,\infty)$:
        \begin{equation}\label{E: kmonotone}
            \ff_k := \left\{ f_{k,G}(\cdot)\equiv \int_0^{\infty} k\frac{(y- \cdot)^{k-1}_{+}}{y^k}\;\ud G(y)\;\bigg|\; G\in{\cal G}_1\right\}\,,
        \end{equation}
        where ${\cal G}_1$ will denote the set of all distribution
        functions on $(0,\infty)$ grounded at 0. Here, we use the
        notation $x_{+}\equiv x\cdot {1}_{[x\geq 0]}$ for any
        $x\in\RR$. It has been shown by \citet{MR0077581} that the family
        $\ff_k$ is identifiably indexed by ${\cal G}_1$. In other words,
        if $G_1, G_2$ are distinct elements in ${\cal G}_1$, then
        $f_{k,G_1}(\cdot)$ and $f_{k,G_2}(\cdot)$ differ on a Lebesgue
        non-null set.  Note that ${\cal F}_k$ is exactly the collection of all scale mixtures of Beta$(1,k)$ densities.

        The $Beta(1,1)$ distribution is the standard uniform distribution, $U(0,1)$.
        Therefore, the class $\ff_1$ coincides with the class of all scale mixtures of uniform densities
        on $(0,\infty)$.
        A well-known theorem by Khintchine (see, e.g., \citet[p.158]{MR0270403})
        asserts that the class of densities on $(0,\infty)$ with concave distribution functions
        is one and the same with our class $\ff_1$. It can be seen that $\ff_1$ is also the class
        of all upper semi-continuous, non-increasing densities on $(0,\infty)$.
        This class is induced by order restrictions,  a term we use to explicitly mean
        that there exists a partial ordering $(\ll)$ on the common support ${\cal X}$ of
        the densities in $\ff_1$ such that $f\in\ff_1$ if and only if $f$ is
        isotone with respect to this ordering: i.e., $f\in\ff_1$ if and only
        if $f(x)\leq f(y)$ whenever $x,y\in{\cal X}$ such that $x\ll y$. In this case,
        $(\ll)$ is the natural partial ordering, $(\geq)$, on
        $(0,\infty)$.

        Non-increasing, upper semi-continuous densities (in short, {\sl monotone densities})
        arise naturally via connections with renewal theory and uniform mixing
        (see, e.g., \citet{MR1243398}.) Maximum likelihood estimation of
        monotone densities on $(0,\infty)$ was initiated by
        \citet{MR0086459,MR0093415}, with related work by \citet{MR0073895},
        \citet{MR0132632} and
        \citet{MR0074746, MR0083859, MR0077040,MR0083869,MR0090196}.
        Asymptotic theory of the MLE in $\ff_1$ (the Grenander
        estimator)
        was developed by \citet{MR0267677} with later contributions by
        \citet{MR822052,MR981568}, \citet{MR902241,MR1026298} and
        \citet{MR1041391}.
        See \citet{bjpsw-2010} for descriptions of the behavior of the Grenander estimator
        at zero.

        Nonparametric estimation in families of  densities described by order
        restrictions goes back at least to the work of \citet{MR0086459,MR0093415},
        \citet{MR0132632,MR0277070} and \citet{MR0211526}, with further development by
        \citet{MR0250404,MR0254995,MR0267681} and \citet{MR548023,MR673648}. Also see
        the books by \citet{MR0326887} and by \citet{MR961262}.
        \citet{MR1365636,MR1345204,MR1464172,MR1673281} addressed estimation
        in various order restricted classes of multivariate densities
        from the perspective of the excess mass approach
        studied previously by e.g., \citet{MR548023,MR673648} and \citet{MR1147099}.
        Polonik shows that (under reasonable assumptions) the MLE in such classes exists
        and coincides with an estimator he constructs and calls {\sl the silhouette}.
        Forcing the elements of the class to be upper semi-continuous,
        the MLE is seen to be unique.
       \citet{MR0132632} also gives a graphical construction of the maximum likelihood estimator,
        and establishes $L_1$-consistency of the MLE.

        In this paper our goal is to extend the notion of ``monotone densities''
        to higher dimensions; i.e., to densities on $\s$
        with $d > 1$. Such an extension is not unique: For example,
        we may consider the family, $\ffb(d)$,  of  ``block-decreasing
        densities'' (a term coined by \citet{MR1994726}) that contains all
        upper-semicontinuous densities on $\s$ that are
        non-increasing in each coordinate, while keeping all other
        coordinates fixed. This class was perhaps first introduced
        by \citet{MR0211526}.
        The particular proper subclass of $\ffb(d)$
        studied here is the family $\ffs(d)$ of all
        multivariate {\sl scale mixtures of uniform densities}; i.e.
        the family of upper semi-continuous
        densities on $\s$ of the form
        \begin{equation}\label{E: SMUF1}
            f_{G}(\m{x}) = \int_{\s} \left( \frac{1}{|\m{y}|}\,
            {1}_{(\m{0},\m{y}]}(\m{x})\right)
            \;\ud G(\m{y}) \ , \ \ \m{x} \in \s
        \end{equation}
        for some $G \in {\cal G}_d$,   the set of all distribution
        functions on $\s$ that grounded (zero) at $\m{0}$ ;  here
         we use the notation
        $|\m{y}| \equiv \prod_{i=1}^d y_i$ for $\m{y}=(y_1,\dots,y_d)'\in\s$.
        For any fixed $G \in {\cal G}_d$, it is clear that if
        $\m{Y} = (Y_1,\dots,Y_d)'$ is distributed according to
        $G$ on $\s$ and if $U_1,\dots,U_d$ are i.i.d. $U(0,1)$
        (and independent of $\m{Y}$), then the vector
        $\m{X} := (U_1 Y_1,\dots,U_d Y_d)$ is distributed
        according to $f_{G}(\cdot)$ on $\s$.

        Whereas the family $\ffb(d)$ is
        characterized by order restrictions (and thus the
        results by Polonik apply), its subclass $\ffs$ is not;
        as will be made more explicit in section 2, densities in the
        class $\ffs$ also satisfy non-negativity restrictions
        on their $d-$dimensional differences around all rectangles.
        Because of this additional {\sl shape restriction},
        estimation in this family requires separate treatment.

        A univariate parallelism to the latter point would
        be to consider the family $\ff_2$ in \eqref{E: kmonotone},
        induced by mixtures of triangular densities;  this class
        can easily be seen to be exactly the class
        of all non-increasing, convex (and hence continuous)
        densities on $(0,\infty)$.   Thus
        $\ff_2\subset\ff_1$ is not an order-constrained class of densities,
        in contrast to its superclass $\ff_1$.   Convex densities arise in connection
        with Poisson process models for bird migration and scale mixtures of
        triangular densities (see, e.g., \citet{30475}, \citet{MR2028914} and
        \citet{LSM91}). Estimation of non-increasing, convex densities on $(0,\infty)$
        was apparently initiated by \citet{Anevski94} and was further pursued by
        \citet{Wang-Y-94}, \citet{Jongbloed/PHD} and \citet{MR2028914}. The
        asymptotic distribution theory and further characterizations of the
        nonparametric MLE of such a density and its first derivative
        at a fixed point (both under reasonable assumptions) was obtained by
        \citet{MR1891741,MR1891742}. These authors show that the local
        rate of convergence of the MLE  of the functional
        $f\mapsto f(x)$ is of the order $n^{2/5}$, whereas the
        Grenander estimator (the MLE in $\ff_1$) converges
        locally at the rate of only $n^{1/3}$.

        Here is an outline of the remainder of the present paper:
        In Section~\ref{S: family}
        we provide characterizations of the family $\ffs(d)$ that
        will prove useful in the sequel. Section~\ref{S: MLE}
        addresses existence, strong, pointwise
        consistency as well as $L_1$ and Hellinger consistency of a sequence
        of maximum likelihood estimators in $\ffs(d)$.
        In Section~\ref{S: minimax} we derive a local asymptotic minimax lower
        bound for estimation of $f(\m{x} )$ at a fixed point $\m{x}$ under for which $f$
        satisfies $\partial^d f (\m{x} )/ (\partial x_1 \cdots \partial x_d )  \not= 0$.  The
        lower bound entails a rate of convergence
        of $n^{1/3}$ for all dimensions $d$ and yields a constant depending on $f$
        which reduces to the known lower bound constant for $d=1$.
        The paper
        concludes in Section~\ref{S: discussion}
        with a discussion of conjectures and open problems
        related with both the local (pointwise) and the global ($L_1$ and
        Hellinger) rates of convergence of the MLE in $\ffs(d)$.

    \section{Properties of the Scale Mixtures of Uniform family of densities}\label{S: family}

    \subsection{Properties of $\ffs(d)$}\label{SS: family}

        A density function, $f$, on $\s$ will be called a (multivariate)
        {\sl Scale Mixture of Uniform densities}  if there exists a
        distribution function, $G$, on $\s$ such that
\begin{eqnarray}
f(\m{x})=f_{G}(\m{x}) &=& \int_{(0,\infty)^d} \frac{1}{|\m{v}|}
            1_{(\m{0},\m{v}]}(\m{x})\;\ud G(\m{v}) \label{E: SMU1} \\
&=& \int_{\m{v}\ge\m{x}}
            \frac{1}{|\m{v}|}\; \ud G(\m{v})\;  \quad\text{for all
            $\m{x} \in (0,\infty)^d$}\,. \label{E: SMU2}
\end{eqnarray}
It is clear from \eqref{E: SMU2} that a SMU density is also a
block-decreasing density: $f_G(\cdot)$ is
        non-increasing in each coordinate, while keeping all other
        coordinates fixed. Also, the map $G \mapsto f_G$ is identifiable in the following sense:
        if $G_1 \not= G_2$, then $f_{G_1} \not= f_{G_2}$ on a set of positive Lebesgue measure;
        also see Theorem~\ref{T:InversionFormulaSMU} below.
        The following lemma gives a formal statement and proof of a slightly more general result.

 \begin{lemma}\label{L: BDD identifiability}
            Two upper semi-continuous and block-decreasing functions
            $f$ and $g$ on $\RR^d$ differ nowhere in the interior of
            their support or else on a Lebesgue non-negligible set.
\end{lemma}

\begin{proof}
            Assume that $\m{x}$ is in the interior of the support of
            both $f$ and $g$ and that $f(\m{x})\neq g(\m{x})$.
            Without loss of generality, assume that
            $f(\m{x})>g(\m{x})$. Since $g$ is upper semi-continuous
            and $\m{x}$ is an element of the $\|\cdot\|_2$-open set
            $\{ \m{y} \mid g(\m{y})<f(\m{y})\}$, we have that there
            exists an $\epsilon>0$ such that the $\|\cdot\|_2$-ball
            of radius $\epsilon$ around $\m{x}$,
            $B_{\|\cdot\|_2}(\m{x},\epsilon)$, be a subset of
            $\{ \m{y} \mid g(\m{y})<f(\m{y})\}$. In fact, we have
            that $f$ and $g$ differ on the Lebesgue non-null set
            $A \equiv \{ \m{y}\leq\m{x} \mid \| \m{x} - \m{y}\|_2 <
            \epsilon\}$ since $y\in A$ implies that $g(\m{y}) <
            f(\m{x}) \leq f(\m{y})$ and subsequently that $g(\m{y})
            < f(\m{y})$ -- where here we have also used the fact that
            $f$ is block-decreasing. The proof is complete. \qedhere
 \end{proof}

        The distribution function $F_G$ corresponding to $\m{X}\sim f_G$ is given by
\begin{eqnarray}
F_G ( \m{x})
 =  \int_{(0,\infty)^d} \frac{|\m{x}\wedge\m{v}|}{|\m{v}|} \; \ud G(\m{v}) \, ,
\end{eqnarray}
where $\le $ denotes the natural partial ordering on $\mathbb{R}^d$,
while
\begin{align*}
 \m{x} \wedge \m{v} & \equiv            (x_1,\dots,x_d)\wedge(v_1,\dots, v_d) =
            (\min\{x_1,v_1\},\dots,\min\{x_d,v_d\}) ,
           \intertext{and}
\m{x} \vee \m{v} & \equiv  (x_1,\dots,x_d)\vee(v_1,\dots,v_d) =
            (\max\{x_1,v_1\},\dots,\max\{x_d,v_d\})\, .
\end{align*}
The distribution function $F_G$ of $\m{X}\sim f_G$ is generally not
        concave
        when $d>1$, unlike the case when $d=1$.
        A SMU density (and a block-decreasing density, in general)
        can possibly diverge at the origin, whereas the pointwise
        bound $f(\m{x})\leq 1/|\m{x}|$ holds since, for $\m{x} \in (0,\infty)^d$ we have
 \[
            1=\int_{(0,\infty)^d} f(\m{y})\; \ud \m{y} \geq
            \int_{(\m{0},\m{x}]} f(\m{y})\; \ud \m{y} \geq
            |\m{x}|f(\m{x})\,.
 \]
Further, a $d-$dimensional analogue of the proof of
        \citet[Theorem~6.2, p. 173]{MR836973}
        can be used to show that
\begin{equation}\label{E:BDD Limits}
            \lim_{|\m{x}|\rightarrow\infty}\{|\m{x}|f(\m{x})\}
            =\lim_{\m{x}\downarrow\m{0}}\{|\m{x}|f(\m{x})\}=0\,,
\end{equation}
whenever $f$ is a block-decreasing density on $\s$.

For any two points $\m{x},\m{y}\in [0,\infty)^d$, such
        that $\m{x} \le \m{y}$, we write
        $[\m{x},\m{y}]\equiv[x_1,y_1]\times\dots\times [x_d,y_d]$,
        $[\m{x},\m{y})\equiv[x_1,y_1)\times\dots\times [x_d,y_d)$,
        $(\m{x},\m{y}]\equiv(x_1,y_1]\times\dots\times (x_d,y_d]$,
        $(\m{x},\m{y})\equiv (x_1,y_1)\times\dots\times (x_d,y_d)$
        for the natural closed, lower-closed upper open, lower open upper closed, and open
        rectangles respectively.   Note that the closed rectangle $[\m{x}, \m{y}]$ has (at most) $2^d$
        vertices, the points $\m{u} = (u_1, \ldots , u_d )$ where each $u_i$ is either $x_i$ or $y_i$.
          Following \citet{MR1324786}, we write $\mbox{sgn}_{[\m{x}, \m{y}]} (\m{u} ) \in \{ -1, 1\}$, the signum of the
          vertex $\m{u}$, according as the number of $i$, $1 \le i \le d$, satisfying $u_i = x_i$ is odd or even respectively.

        Thus any two
        vertices defining an edge of the rectangle  have
        alternating signs. Then, if
        $\m{u}=(u_1,\dots,u_d)$ is some vertex of
        $[\m{x},\m{y}]$
        and $\delta\in\{-1,+1\}$ is its signum, then
        $(\delta,\m{u})$ is an element of the set
\begin{equation*}
            \Delta_d [\m{x},\m{y}]  =  \left\{\left(
            (-1)^{\sum_{i=1}^{d}\left\{\mathbbm{1}_{[u_i=x_i]}\right\}}\;,\;
            \m{u} \right)\; \Bigg|\;
            \m{u}\in  \{x_1,y_1\} \times \cdots \times \{ x_d , y_d \} \right\}\, .
\end{equation*}
\begin{definition} \label{D: fvolume}
            For an upper semicontinuous and coordinatewise decreasing
            function $g :  \s \rightarrow [0,\infty)$ define the
            $g$-volume of a (possibly degenerate) rectangle $[\m{x},\m{y})$ by:
            \begin{eqnarray}
                V_{g}[\m{x},\m{y}) = \sum_{\left(\delta,\m{u}\right)
                \in \Delta_{d}[\m{x},\m{y}]} \left\{ \delta g(\m{u}) \right\}\,,
                \label{SumOverVerticesgVolumeFormula}
            \end{eqnarray}
            provided that $g$ is defined and is finite for all
            $\m{u}$ in the summand.   Correspondingly, for an upper semicontinuous and coordinatewise
            increasing function $g : \s \rightarrow [0,\infty)$, we define the $g$-volume of a rectangle $(\m{x}, \m{y}]$
            by the sum on the right side of (\ref{SumOverVerticesgVolumeFormula}).
\end{definition}

It is easily seen that for a SMU density, $f_G$, the
        $f_G$-volume of any rectangle $[\m{x},\m{y})$ is always of
        the sign $(-1)^d$: Indeed, consider \eqref{E: SMU2} and
        observe that
\begin{eqnarray}
            (-1)^d V_{f_G}[\m{x},\m{y}) = \int_{[\m{x},\m{y})}
            \frac{1}{|\m{v}|}\,\ud G(\m{v}) \geq 0\,.
\label{VsubfsubGIsNonNeg}
\end{eqnarray}
From (\ref{VsubfsubGIsNonNeg}), or, alternatively, from the fact that
the class of sets $[\m{x}, \m{y})$ is a $\pi-$system which generates
the Borel $\sigma-$field of subsets of $[0,\infty)^d$ and then
extending as in \citet{MR1324786}, it is clear that $(-1)^d V_f$
extends uniquely to a (non-negative) measure on the Borel
$\sigma-$field ${\cal B}_+^d = {\cal B}^d \cap [0,\infty)^d$ given by
\begin{eqnarray*}
(-1)^d V_f (A) = \int_A \frac{1}{|\m{v}|} dG( \m{v} )\qquad
\mbox{for} \ \ A \in {\cal B}_+^d ;
\end{eqnarray*}
in particular,
\begin{eqnarray*}
(-1)^d V_f (\m{x} , \m{y}] = \int_{(\m{x} , \m{y}]} \frac{1}{|\m{v}|}
dG( \m{v} ) .
\end{eqnarray*}
The following lemma extends this argument to an arbitrary upper
semicontinuous function $g$ with the $(-1)^d g-$volumes of all
rectangles $[\m{x}, \m{y} )$ non-negative.

\begin{lemma}
\label{L: differencing} Suppose that $g$ is a non-negative, upper
semi-continuous function satisfying $(-1)^d V_g [\m{x}, \m{y}) \ge 0$
for all lower-closed upper open rectangles $[\m{x}, \m{y})$, and
vanishing if any coordinate tends to $\infty$.   Then $(-1)^d V_g$
can be extended to a countably additive measure on ${\cal B}_+^d$.
\end{lemma}

\begin{proof}   Since the class of all rectangles of the form $[\m{x}, \m{y})$  is a $\pi-$system
which generates ${\cal B}_+^d$, this follows immediately from the
analogue of \citet*{MR1324786} with obvious modifications (replace
Billingsley's sets $A$ with our sets $[\m{x}, \m{y})$ and  $F$ with
$\bar{F} (\m{x} ) = V_g [\m{x}, \m{\infty} )  $ continuous from
below). \qedhere
\end{proof}

Of course it is easy to exhibit a block-decreasing density that is
not a SMU density: consider the uniform density on the closed
 triangle in $\RR_+^2$ with vertices $(0,0)$, $(0,1)$ and $(1,0)$. Then,
$$
(-1)^2 V_{f}[({1}/{8},{1}/{8}),({1}/{2},{3}/{4})) = -2  < 0\,,
$$
showing that this density is not a SMU density,  even though it is
block-decreasing.

The following theorem establishes identifiability of the mixing
distribution $G$ as well as providing a useful characterization of
SMU densities.
\newpage

\begin{theorem}
\label{T:InversionFormulaSMU} \hfill
   \begin{enumeratea}
   \item For the class of SMU densities $\ffs(d) =\{ f_G :  \ G \in {\cal G}_d \}$
   with $f_G$ as given in (\ref{E: SMU1}),
$f \in \ffs(d)$ if and only if $f\equiv f_G$, where $G\in{\cal G}_d$
is given by
\begin{equation}\label{E: SMU correspondence}
                    G(\m{x}) = \int_{\s} (-1)^{d}
                    V_{f}(\m{u},\m{x}]\cdot\mathbbm{1}_{[\m{u}\le\m{x}]}\; \ud\m{u}\,.
\end{equation}
Thus there is a one-to-one correspondence between $G \in {\cal G}_d$
and $f_G \in \ffs(d)$.

\item
Suppose that the Lebesgue density $f$ on $\s$ is such that it
converges to zero in each coordinate, while keeping all other
coordinates fixed. Then, $f$ is a SMU density if and only if
$(-1)^{d}V_{f}[\m{x},\m{y})\geq 0$ for all $\m{0} \le \m{x} \le
\m{y}$.
\end{enumeratea}
\end{theorem}

\begin{proof}
(a) \  Suppose that $f\equiv f_G$, for $G\in{\cal G}_d$ (recall that
        this implies that $G(\m{0})=0$), is a SMU
        density evaluated at an arbitrary $\m{x}\in\s$ as:
        \begin{equation}\label{E: SMUru1}
            f(\m{x}) = \int_{\s}
            \frac{1}{|\m{y}|}\mathbbm{1}_{(\m{0},\m{x}]}\;
            \ud G(\m{y})
            =  \int_{{y_1}\geq {x_1}}\cdots\int_{y_d\geq x_d}
            \frac{1}{|\m{y}|}\; \ud G(\m{y})\,,
        \end{equation}
        so that
        $\ud f(\m{x}) = (-1)^{d} |\m{x}| ^{-1} \; \ud G(\m{x})$
        and thus,
\begin{eqnarray*}
G\left(\m{x}\right)
& = & \int_{\s} \mathbbm{1}_{(\m{0},\m{x}]} ( \m{y})   |\m{y}| \; \ud\{(-1)^{d}f(\m{y})\} \\
& = & \int_{(\m{0},\m{x}]} \int_{(\m{0},\m{x}]}
            \mathbbm{1}_{(\m{0},\m{y}]}(\m{u})\; \ud\m{u}\,
            \ud\{(-1)^{d}f(\m{y})\} \\
& = & \int_{(\m{0},\m{x}]} \left \{ \int_{\m{y} \in (\m{u},\m{x}]}\;
            \ud\{(-1)^{d}f(\m{y})\} \right \} \, \ud\m{u} \\
& = & \int_{(\m{0},\m{x}]}(-1)^{d}V_{f} (\m{u},\m{x}]\;
            \ud\m{u}\,,
\end{eqnarray*}
        where the second to last equality follows by Fubini-Tonelli.

        We will now show that $G$ is unique: Suppose that
        \eqref{E: SMUru1} above holds for $G = G_i
        \in {\cal G}_d$ and $i = 1,2$. Recall that this
        implies that $G_1(\m{0})=G_2(\m{0})=0$ and, thus,
        $G_0(\cdot)\equiv G_1(\cdot) - G_2(\cdot)$ is such
        that $G_0(\m{0})=0$, $\int_{\s} G_0(\m{x})\;\ud\m{x} = 0$
        and
        \begin{equation}\label{E: SMUru2}
            0 = \int_{\s}
            \frac{1}{|\m{y}|}\mathbbm{1}_{(\m{0},\m{x}]}\;
            \ud G_0(\m{y}) = \int_{(\m{0}, \m{x}]}
            \frac{1}{|\m{y}|}\ud G_0(\m{y})
        \end{equation}
        holds for all
        $\m{x}\in\s$ and, thus, necessarily $G_0(\m{x})$ has to
        be independent of $\m{x}$ and therefore everywhere equal
        to its value at $\m{0}$: $G_0(\m{0}) = 0$. This completes
        the assertion of uniqueness, since $G_1\equiv G_2$.

(b) \  If $f$ is in ${\cal F}_{SMU}$, there exists
        $G\in {\cal G}_d$ such that
        \begin{eqnarray*}
            f(\m{x})  =  \int_{\s} \frac{1}{|\m{y}|}
            \mathbbm{1}_{(\m{0},\m{y}]}(\m{x})\; \ud G(\m{y})
             =  \int_{\m{y}\geq \m{x}} \frac{1}{|\m{y}|}\;
            \ud G(\m{y})\,,
        \end{eqnarray*}
        so that it is easily seen that
        $(-1)^{d}V_{f}[\m{x},\m{y})=\int_{[\m{x},\m{y})}
        {|\m{y}|}^{-1}\; \ud G(\m{y}) \geq 0$ holds true for all
        $\m{0}\le\m{x}\le\m{y}$.

        On the other hand, assume that the Lebesgue density $f$
        is such that it converges to zero
        in each coordinate, while keeping all other coordinates
        fixed, and satisfies $(-1)^{d}V_{f}[\m{x},\m{y}]\geq 0$
        for all $\m{0}\le\m{x}\le\m{y}$. First, observe that,
        by Lemma~\ref{L: differencing}, this implies that for
        $\m{x}_1\le\m{x}_2\le\m{x}$, elements of $\s$, we have
        \[
            (-1)^d V_f[\m{x}_1,\m{x}) \geq (-1)^d V_f[\m{x}_2, \m{x})
        \]
        and, letting $\m{x}\rightarrow \infty$, this yields
        $f(\m{x}_1)\geq f(\m{x}_2)$ because we assumed that $f$
        vanishes as any one of its coordinates diverges to infinity,
        so that $V_f[\m{x}_i,\m{x})\rightarrow (-1)^df(\m{x}_i)$ for
        $i\in\{1,2\}$. Thus, $f$ is block-decreasing.

        Hence, by appealing to part (i), it thus suffices to show that
        $G$, as defined on $\s$ by (\ref{E: SMU correspondence})
        is a valid distribution function. \\
\medskip
(i) \  $G$ is grounded at $\m{0}$ trivially by inspection:  $G(\m{0})=0$.\\
(ii) \ Notice that
  \begin{eqnarray*}
    \lefteqn{\lim_{x_1\wedge\cdots\wedge x_d\rightarrow \infty} G(x_1,\dots,x_d)  =
                        \lim_{n\rightarrow\infty}\left\{G(n\m{1})\right\} } \\
   & = & \lim_{n\rightarrow\infty} \int_{\s}
                (-1)^{d}V_{f}(\m{u},n\m{1}]\, \mathbbm{1}_{[\m{u} \le
                n\m{1}]}\; \ud\m{u} \\
 & = & (-1)^{d} \int_{\s}
                \lim_{n\rightarrow\infty}\left\{V_{f} (\m{u},n\m{1}]\right\}\,
                \lim_{n\rightarrow\infty}\left\{\mathbbm{1}_{[\m{u} \le
                n\m{1}]}\right\}  \ud\m{u} \\
& = & (-1)^{d} \int_{\s} (-1)^{d}f(\m{u}+)\; \ud\m{u} =  \int_{\s}
f(\m{u})\; \ud\m{u}
                = 1\,,
\end{eqnarray*}
            where in the steps above we have used the fact that for
            each fixed $\m{u} \in\s$, the sequence $X_n(\m{u}) :=
            V_{f}(\m{u},n\m{1}]\, \mathbbm{1}_{ [\m{u} \le n\m{1}]}$ is
            increasing in $n \in \NN$ and we applied the
            monotone convergence theorem, and noted that
            $\lim_{n\rightarrow\infty}\{\mathbbm{1}_{[\m{u} \le n \m{1}]}\}=1$
            for any fixed $\m{u} \in \s$, and that
            $$
            \lim_{n\rightarrow\infty}\{V_{f}(\m{u},n\m{1}]\} =
            \lim_{n\rightarrow\infty}\sum_{(\delta,\m{v})\in
            \Delta_{d}[\m{u},n\m{1}]} \delta f(\m{v} ) = (-1)^{d}f(\m{u} +)
            $$
            because
            $$0\leq\lim_{|\m{x}|\rightarrow\infty}f(\m{x})\leq
            \lim_{|\m{x}|\rightarrow\infty}\{1/{|\m{x}|}\}=0\,,
            $$
            since $f$ is block--decreasing. Finally, the proof
            is complete as soon as we observe that $(-1)^{2d}=1$ and that
            $\int_{\s}f(\m{u})\; \ud\m{u} = 1$, since $f$ is a
            density.\\
(iii) \ Now, fix $\m{0}\leq\m{x}\leq\m{y}$ and
            note that (since $G$ is an {\sl increasing} upper-semicontinuous function)
\begin{eqnarray*}
V_{G}(\m{x},\m{y}] & = & \sum_{(\delta,\m{v}) \in \Delta_d
[\m{x},\m{y}]}
                \{\delta G(\m{v})\} \\
& = & (-1)^{d} \int_{\s}\sum_{(\delta,\m{\epsilon})
                \in \Delta_{d}[\m{x},\m{y}]}\{\delta
                V_{f}(\m{u},\m{v}]\,
                \mathbbm{1}_{[\m{u} \le \m{v}]}\}\; \ud\m{u} \\
& = & \int_{(\m{0}, \m{y}]}
                (-1)^{d}V_{f}(\m{u}\vee\m{x},\m{y}]\; \ud\m{u} \geq
                0\,,
            \end{eqnarray*}
            by geometric inspection
            and  Lemma~\ref{L: differencing}.
            \qedhere
\end{proof}

\subsection{Lebesgue measurability of block-decreasing functions}\label{SS: easurability}
Now we establish a technical fact concerning the (Lebesgue)
measurability of block-decreasing functions which will be needed in
our proofs in Section~\ref{SS: consistency}. We begin with a
definition and then a lemma.
        \begin{definition}\label{D: Defective rectangle}
            \emph{We call a subset $C$ of $\RR^d$ a \emph{``defective rectangle''}
            if and only if there exist real numbers $a_i < b_i$ for $i =
            1,2,\dots,d$, such that
            \[
                (a_1,b_1)\times\cdots\times (a_d,b_d)\subseteq C\subseteq
                [a_1,b_1]\times\cdots\times [a_d,b_d]\; .
            \]
            Thus, by definition, a \emph{defective rectangle} is a
            compact rectangle in $\RR^d$ \emph{minus} a potentially
            non-void subset of its \emph{boundary}. In our definition, a
            \emph{defective rectangle} is taken to be both \emph{bounded}
            and \emph{non-degenerate}.}
        \end{definition}

        \begin{lemma} \label{L: DOUL}
        Any union of defective rectangles in $\RR^d$ is a Lebesgue set.
        \end{lemma}

        \begin{proof}
            Let $\mathcal{C}=\{C_j \mid j\in J\}$ be a family of defective
            rectangles in $\RR^d$, indexed by some set $J$.
            For each $j\in J$ let the real numbers $a_{i,j} <
            b_{i,j}$, for $i\in\{1,2,\dots,d\}$, be uniquely determined by
            \[
                (a_{1,j},b_{1,j})\times\cdots\times (a_{d,j},b_{d,j})
                \subseteq C_j \subseteq
                [a_{1,j},b_{1,j}]\times\cdots\times [a_{d,j},b_{d,j}]\; .
            \]
            For any $\m{x}\in \RR^d$ and $\epsilon > 0$ let
            $B(\m{x},\epsilon)$ denote the open $\|\cdot\|_2$-ball centered at
            $\m{x}$ and with radius less than $\epsilon$. Let also
            $\lambda^{\ast}$ denote outer-Lebesgue measure on $\RR^d$
            and $\lambda$ its restriction on the Lebesgue sets.

            Let $\Delta \equiv \bigcup_{j\in J} C_j$ denote the union of the
            elements in $\mathcal{C}$ and notice that the \emph{interior}
            subset of $\Delta$ is the set
            \[
                \operatorname{int}(\Delta) = \bigcup_{j\in J}
                (a_{1,j},b_{1,j})\times\cdots\times (a_{d,j},b_{d,j})\;,
            \]
            exactly because $\operatorname{int}(C_j) =
            (a_{1,j},b_{1,j})\times\cdots\times (a_{d,j},b_{d,j})$ for each
            $j \in J$ and because an arbitrary union of open sets is open.
            Since $\operatorname{int}(\Delta)$ is an open set, to show that
            $\Delta$ is a Lebesgue set, it suffices to show that
            $\lambda^{\ast}(\Delta\bsl \operatorname{int}(\Delta)) = 0$,
            from which one concludes that $\Gamma \equiv\Delta\bsl
            \operatorname{int}(\Delta)$ is a Lebesgue-null set
            and hence $\Delta$ a Lebesgue set also.

            Notice that if $\Gamma = \emptyset$ there is nothing to
            show. Now, given $\Gamma \neq\emptyset$, fix an arbitrary element
            $\m{y}\in\Gamma$ and observe that there exists an index $k\in J$
            such that $\m{y}$ lies on the boundary of $C_k$; i.e.,
            $\m{y}\in \partial\operatorname{cl}(C_k)$ where
            $\lambda(\operatorname{cl}(C_k)) =
            \prod_{i=1}^{d} (b_{i,k}-a_{i,k}) > 0$. Letting
            \[
                V_{C_k} \equiv \{a_{1,k},b_{1,k}\}\times\cdots\times
                \{a_{d,k},b_{d,k}\}
            \]
            denote the $2^d$ vertices of $\operatorname{cl}(C_k)$ we have
            that
            \[
                \frac{\lambda(\operatorname{int}(C_k)\cap B(\m{y},\epsilon))}
                {\lambda(B(\m{y},\epsilon))}\geq \left(\frac{1}{2}\right)^d
            \]
            holds true for all $0 < \epsilon < \min\{\|\m{y} - \m{z}\|_2
            \mid \m{z}\in V_{C_k}\bsl\{\m{y}\}\}$. This observation, in conjunction
            with the fact that $\operatorname{int}(C_k)\subseteq
            \Gamma^{c}$, immediately yield
            \[
                \varlimsup_{\epsilon \downarrow
                0}\left\{\frac{\lambda^{\ast}\left(\Gamma\cap
                B(\m{y},\epsilon)\right)}{\lambda\left(B(\m{y},\epsilon)
                \right)}\right\}\leq 1 - \left(\frac{1}{2}\right)^d < 1\;.
            \]
            The last inequality, and the fact that $\m{y}\in\Gamma$ was
            arbitrary, show (by appealing to the Lebesgue density
            theorem, see e.g. \citet[Corollary 6.2.6, pg. 184]{MR578344}) that $\Gamma$ contains no density points and
            is consequently a Lebesgue-null set.
            \qedhere
        \end{proof}

        With this lemma at hand we are ready to prove Lebesgue measurability
        of non-negative, block-decreasing functions that vanish at
        infinity.

        \begin{proposition}\label{P: LMBDF}
            Let $f$ be a real-valued, non-negative function on $(0,\infty)^d$
            that is non-increasing and convergent
            to zero in each coordinate $x_j$, keeping all other coordinates fixed, as $x_j$
            coordinate tends to $\infty$.  Then:
            \begin{enumeratea}
                \item $f$ is Lebesgue-measurable.
                \item There exists such a function $f$ that is not
                Borel-measurable. Such an $f$ exists with $f$ also satisfying
                $\sup\{f(\m{x})\mid \m{x}\in (0,\infty)^d\} < \infty$.
            \end{enumeratea}
        \end{proposition}

        \begin{proof}
        Proposition~\ref{P: LMBDF} follows from Theorem 3 of \cite{MR855142}, but for
        completeness we give another proof here.
(a) \ Note that $[f\geq 0]\equiv [\m{x}\in (0,\infty)^d
                \mid f(\m{x})\geq 0]$,  the support of $f$, is the  closure of
                $[\m{x}\in (0,\infty)^d \mid f(\m{x})>0]$, and thus a Borel set;
                hence it is also a Lebesgue set.

                Fix $t > 0$;
                since $f$ is non-negative,
                block-decreasing and vanishes at infinity,
                $[f\geq t]\equiv [\m{x}\in (0,\infty)^d
                \mid f(\m{x})\geq t]$ has the form
                \[
                    [f\geq t] = \bigcup_{\m{x}\in A_t} C_{\m{x}}
                \]
                for some (non-unique) subset $A_t$ of $(0,\infty)^d$,
                where
                \[
                    C_{\m{x}} \in \left\{ (\m{0},\m{x}],
                    (\m{0},\m{x}]\bsl\{\m{x}\}\right\}
                \]
                is a defective rectangle (by
                Definition~\ref{D: Defective rectangle}), for
                each $\m{x}\in A_t$. Hence
                it follows by Lemma~\ref{L: DOUL} that
                $[f\geq t]$ is
                a Lebesgue set. Since the argument above holds for
                all $t>0$, the proof of
                Lebesgue-measurability of $f$ is complete since
                the class of sets $\{[t,\infty)\mid t\in\RR\}$
                generates the Borel $\sigma$-field.

(b) \  We shall provide a counter-example in two dimensions, $d=2$.
                For higher dimensions, analogous counter-examples can be constructed.
                As soon as we convince ourselves that a non-Borel subset, $A$, of
                $\Delta \equiv \{(x,1-x) \in (0,1)^2 \mid 0<x<1\}$
                exists, we construct $f$ on
                $(0,\infty)^2$, satisfying
                $\sup\{f(\m{x})\mid \m{x}\in (0,\infty)^2\} <\infty$,
                by $f(\cdot) \equiv \mathbbm{1}_{\tilde{A}}(\cdot)$ where
                \[
                    \tilde{A}\equiv \bigcup_{(x,y)\in A} (0,x]\times
                    (0,y]\;.
                \]
                Notice then that $[f\geq 1] = \tilde{A}$ is \emph{not} a
                Borel set as $A$ is taken to be a non-Borel subset of
                $\Delta$ and it is an easy task to verify that
                $\Delta\cap\tilde{A} = A$. Indeed, on one hand $A\subseteq
                \Delta\cap\tilde{A}$ follows directly from
                $A\subseteq\tilde{A}$ and $A\subseteq\Delta$. On the other
                hand, if $(x,y)\in \Delta\cap\tilde{A}$ we have that there
                exists an $(x_0,y_0)\in A$ such that
                \begin{gather*}
                    0<x, x_0, y_0, y<1\;, \\
                    x+y = x_0+y_0 = 1\;, \\
                    x\leq x_0\;\text{ and }\; y\leq y_0\;.
                \end{gather*}
                Combining the above relationships we conclude that
                necessarily $(x,y)=(x_0,y_0)\in A$ and the proof of
                $\Delta\cap\tilde{A} = A$ is complete.

                To conclude this counter-example we elaborate briefly
                on the existence of a non-Borel subset $A$ of $\Delta$. In
                doing so, we follow steps as in \citet{MR1762415}.
                Let $D$ be a subset of $(0,1)$ that is not a Lebesgue set -- the
                existence of which is guaranteed by Proposition
                1.2.2 in \citet{MR1762415}. As in Example 7.1.1 of
                \citet{MR1762415}, let $F$ be the Lebesgue singular distribution
                function that gives mass $1$ and is 1--1 on the
                Cantor set, $C$. Let $B=F^{-1}(D)$ so that $B$ be a subset
                of the Cantor set, $C$, and a Lebesgue-null set as $B\subseteq C$ and
                $\lambda(C)=0$. Let also $A \equiv \{(x,1-x)\mid x\in B\}$.
                We argue that $A$ so constructed is not a Borel subset of
                $\RR^2$. Assume the contrary, i.e. assume that $A$ is in fact
                a Borel set. Since the vector-valued function $x \mapsto (x,1-x)$
                is a one-to-one, $(\text{Borel})^2$-measurable mapping on $(0,1)$ we have
                immediately that $B$ must also be a Borel set in $\RR$. But
                then, since $F$ is non-decreasing, we have that $F(B)$ is
                also a Borel set. In addition, since $F$ is one-to-one on $C$, we
                have that $D=F(B)$ and thus that $D$ is a Borel and hence a
                Lebesgue set. This is a contradiction, because $D$ was taken
                to be a non-Lebesgue set, by definition. This contradiction
                yields that $A,$ so constructed, is indeed a non-Borel subset
                of $\RR^2$.\qedhere
        \end{proof}

\section{Existence and Consistency of the MLE}\label{S: MLE}

Let $\m{X}_1,\dots, \m{X}_n$ be i.i.d. random vectors distributed
according to some density $f_0 = f_{G_0}\in\ffs(d)$ where $f_0$ is
unknown.
        Our goal is to estimate the unknown SMU density, $f_0$,
        based on $\m{X}_1,\dots, \m{X}_n$.
We will be interested in
        maximizing the likelihood function $f\mapsto \prod_{i=1}^n
        f(\m{X}_i)$ or, equivalently, the log-likelihood function
$f\mapsto  n \PP_n \log\{f(\m{X})\}$ over $f \in \ffs(d)$ where
$\PP_n = n^{-1} \sum_{i=1}^n \delta_{\m{X}_i}$ is the empirical
measure of the data.
        Any such maximizer, $\widehat{f}_n\in\ffs(d)$, should one exist,
        will be called a (nonparametric) \emph{maximum likelihood estimator} of
        $f_0$, based on $\m{X}_1,\dots, \m{X}_n$.   Since $f_0 = f_{G_0} $ is given by
        (\ref{E: SMU1}) it follows from Theorem~\ref{T:InversionFormulaSMU}
         that estimation of $f_0 \in {\cal F}_{SMU}$ is equivalent
        to estimation of $G_0$.

    \subsection{On existence and uniqueness of an MLE}
    \label{SS: exist MLE}

        We begin with a definition followed by the main theorem of this subsection.

        \begin{definition}\textbf{[Rectangular grid generated by
        data]}\label{D: grid}
            \emph{Suppose that $\m{x}_1,\dots,\m{x}_n$ are
            (fixed or random) elements in $\s$ and suppose that $\m{x}_i=(x_{i1},
            \ldots,x_{id})'$ where $i=1,2,\ldots,n$. Define the matrix $A=[x_{ij}] \in M_{n\times
            d}( (0, \infty))$ whose $i^\text{th}$ row is exactly $\m{x}'_i$,
            for $i\in\{1,2,\dots,n\}$. Also let \\
            $A^{\sharp} = \{ \;
            (x_{(i_1),1}, x_{(i_2),2},\ldots, x_{(i_d),d}) \mid i_1,\ldots,i_d
            \in \{1,2,\ldots,n\}\}$ denote \emph{the rectangular grid generated
            by} $A$, where $x_{(i),j}$  denotes the $i^{th}$ smallest element
            among $x_{1j},\ldots,x_{nj}$ where $i \in \{1,2,\ldots,n\}$ and $j
            \in \{1,2,\ldots,d\}$. In particular, $\m{x}_{*}=(x_{(1),1},
            x_{(1),2},\ldots, x_{(1),d})$ and $\m{x}^{*}=(x_{(n),1},
            x_{(n),2},\ldots, x_{(n),d})$ denote the element-wise minimum and
            maximum of $\m{x}_1,\ldots,\m{x}_n$, respectively.
            For each fixed $j\in\{1,2,\dots,d\}$,
            let \\ 
            $n_j(A) := \operatorname{card}(\left\{x_{i,j} \mid
            i=1,2,\dots,n\right\})$, and notice that we have:
            $\operatorname{card}(A^{\sharp})= \prod_{j=1}^d n_j(A)
            \equiv N \leq n^d$.}
        \end{definition}

 \begin{theorem}\textup{\textbf{[Existence and characterization of an
        MLE in $\ffs(d)$]}}\label{T: MLE}
\begin{enumeratea}
\item
A maximum likelihood estimator (MLE), $\widehat{f}_n\equiv
f_{\widehat{G}_n}\in\ffs(d)$ of $f_0\equiv f_{G_0}\in\ffs(d)$ almost
surely exists, where $\widehat{G}_n\in{\cal G}_d$ is a purely-atomic
probability measure, with at most $n$ atoms, all of which are
concentrated on $A^{\sharp}$ -- the rectangular grid generated by the
data $\m{X}_1,\dots,\m{X}_n$.

\item
For almost all $\omega$, the unique MLE, $\widehat{f}_{n}\equiv
f_{\widehat{G}_{n}}\in\ffs(d)$, is completely characterized by the
following Fenchel conditions:
\begin{align}
& && \PP_n  \left\{ \frac{ \mathbbm{1}_{[\m{X}
                    \le \m{x}]}}{\widehat{f}_{n}\left(\m{X}  \right)}\right\} \leq |\m{x}|\,;\quad
                    \text{for all $\m{x}\in\s$}\,,\label{E: MLESMU3}\\
                    &\text{and }
   && \PP_n \left\{ \frac{ \mathbbm{1}_{[\m{X}
                    \le \m{y}]}}{\widehat{f}_{n}\left(\m{X}  \right)}\right\}
                    = |\m{y}|\,;\quad\text{if and only if}\label{E: MLESMU4}\\
                    & &&\text{$\m{y}\in\s$ satisfies
                    $\widehat{G}_{n}(\{\m{y}\}) > 0$; or, equivalently,}\notag \\
                    & && (-1)^d \lim_{\epsilon \downarrow 0} \left\{ V_{\widehat{f}_{n}}
                    \left[ \m{y}, \m{y} + \epsilon\m{1}\right) \right\} > 0\,.
                    \notag
\end{align}
\end{enumeratea}
\end{theorem}

Maximum likelihood estimation in mixture models has been studied in
general by
 \citet{MR684866}, and this material is nicely summarized in
\citet[Chapter~5]{Lin}. To prove the present theorem, we will
therefore appeal to the results
 in \citet[Chapter~5]{Lin} and \citet{MR0274683}.  We begin with three lemmas.\\
 \smallskip

 \begin{lemma}
\label{P:MLE-ConcOnDataGrid} The support set of the mixing measure
$\widehat{G}_n$ of any MLE $\widehat{f}_n$ is contained in the grid
$A^{\#} \subset (0,\infty)^d$ generated by the observed data
$\m{X}_1, \ldots , \m{X}_n$; i.e. $\mbox{supp} (\widehat{G}_n )
\subset  A^{\#}$.
\end{lemma}
\smallskip

\begin{proof}
First we show that ${\cal Y} \subset (\m{0}, \m{X}^*]$ where $\m{X}^*
\equiv \m{X}_1 \vee \cdots \vee \m{X}_n$ and the maximums are taken
coordinatewise. If $\widehat{f}_n$ maximizes $L_n (f) = n\PP_n \log
f(X)$ over $f \in \ffs(d)$ and there is some $y \in (0,\infty)^d
\setminus (\m{0}, \m{X}^*]$ with $y \in {\cal Y}$, then
$\widehat{f}_n (y)>0$. Since $\widehat{f}_n$ is block decreasing,
this implies that $0 < \int_{(\m{0}, \m{X}^*]} \widehat{f}_n (\m{x})
d \m{x} \equiv \beta < 1$. Then consider $\tilde{f} (\m{x}) \equiv (
\widehat{f}_n (\m{x})/\beta) 1_{(\m{0}, \m{X}^*]} (\m{x})$; it is
easily seen that $\tilde{f} \in \ffs(d)$ and has greater likelihood
than $\widehat{f}_n$, contradicting the assumption that
$\widehat{f}_n$ maximizes the likelihood. Thus ${\cal Y} \subset
(\m{0}, \m{X}^*]$, and we may restrict attention to the class of
estimators with support contained in $(\m{0} , \m{X}^*]$, say ${\cal
K}^* (d)$.  Suppose that $\widehat{f}_n \in {\cal K}^* (d)$. Consider
the mixing measure $\tilde{G}_n$ defined by
\begin{eqnarray*}
\tilde{G}_n \equiv \sum_{j: \m{W}_j \in A^{\#}} \pi_j
\delta_{\m{W}_j} \bigg / \sum_{j: \m{W}_j \in A^{\#}} \pi_j \equiv C
\sum_{j: \m{W}_j \in A^{\#}} \pi_j \delta_{\m{W}_j}
\end{eqnarray*}
where
\[
\pi_j  \equiv (-1)^d V_{\widehat{f}_n } [ \m{W}_j , \m{W}_j^+ ) \cdot
| \m{W}_j |,  \qquad \mbox{for} \ \ \m{W}_j \in A^{\#}
\]
where $\m{W}_j^+ \in A^{\#}$ defines the smallest rectangle above and
right of $\m{W}_j$ in the partition of $[\m{0}, \m{X}^*]$ defined by
the data. Then it is easy to see that
\begin{eqnarray*}
\tilde{f} ( \m{x} ) = \int_{(0,\infty)^d} \frac{1}{|\m{u}|}
1_{(\m{0}, \m{u}]} (\m{x} ) d \tilde{G}_n (\m{u} )
\end{eqnarray*}
satisfies
\begin{eqnarray*}
\tilde{f}( \m{W}_j )
& = & C \sum_{k : \ \m{W}_k \ge \m{W}_j } \frac{\pi_j}{| \m{W}_j |} \\
& = & C  \sum_{k : \ \m{W}_k \ge \m{W}_j } \{ (-1)^d V_{\widehat{f}_n} [ \m{W}_j, \m{W}_k ) \\
& = & C (-1)^d V_{\widehat{f}_n } [ \m{W}_j, 2 \m{X}^* ) = C
\widehat{f}_n ( \m{X}_j ) ,
\end{eqnarray*}
and this implies that
\begin{eqnarray*}
\tilde{f} ( \m{x} ) = C \sum_{j : \m{W}_j \in A^{\#}} 1_{(\m{W}_j^-,
\m{W}_j ]} (\m{x} )
\end{eqnarray*}
where $\m{W}_i^-$ defines the smallest rectangle below and to the
left of $\m{W}_j$ in the partition of $[\m{0},\m{X}^*]$ defined by
the data. If $\widehat{f}_n \not= \tilde{f}$, then there exists
$\m{y} \in ( \m{W}_j^-, \m{W}_j]$ for some $\m{W}_j \in A^{\#}$ such
that $\widehat{f}_n (\m{y}) \not= \tilde{f} (\m{y})$, and then
necessarily $\widehat{f}_n (\m{y}) > \tilde{f} (\m{y}) = \tilde{f} (
\m{W}_j )$. This yields, since $\tilde{f}_n \in {\cal K}^* (d)$,
\begin{eqnarray*}
1 & = & \int_{(\m{0}, \m{X}^*]} \tilde{f} (\m{x}) d \m{x} = C \sum_{j
: \ \m{W}_j \in A^{\#}} \left \{ \widehat{f}_n (\m{W}_j)
              \int_{( \m{W}_j^- , \m{W}_j ]} d \m{x} \right \}  \\
& < & C \sum_{j : \ \m{W}_j \in A^{\#}}  \widehat{f}_n (\m{W}_j)
              \int_{( \m{W}_j^- , \m{W}_j ]} \widehat{f}_n ( \m{x} ) d \m{x}
              = C \int_{(\m{0}, \m{X}^*]} \widehat{f}_n (\m{x}) d \m{x}
 = C
 \end{eqnarray*}
since $f \in {\cal K}^* (d)$.   Thus $\tilde{f}$ has a greater
log-likelihood than $\widehat{f}_n$, and it follows that $\mbox{supp}
(\widehat{G}_n ) \subset A^{\#}$. \qedhere
\end{proof}
\smallskip

Now we can prove uniqueness of the MLEs $\widehat{f}_n$ and
$\widehat{G}_n$.
\smallskip

\begin{lemma}
  \label{P:MLE-existenceAndStructure}
  There exists a set of points ${\cal Y} = \{ \m{y}_1, \ldots , \m{y}_m \} \subset (0,\infty)^d$ with $m \le n$
  such that a $\ffs(d)$ density $\widehat{f}_n$ with corresponding mixing measure $\widehat{G}_n$ is the MLE  only
  if $\mbox{supp} (\widehat{G}_n) \subset {\cal Y}$.  Thus any MLE has the form
  \begin{eqnarray}
     \widehat{f}_n (\m{x}) = \sum_{j=1}^m \pi_j \frac{1}{ | \m{y}_j |} 1_{(\m{0} , \m{y}_j ]} ( \m{x} )
     \label{StructureOfMLEStepOne}
  \end{eqnarray}
 where $\pi_j \ge 0$, $\sum_{j=1}^m \pi_j = 1$.  Moreover, the vector $( \widehat{f}_n (X_i) )_{i=1}^n$ is unique.
 \end{lemma}

 \begin{proof}  As in \citet{MR684866,Lin},
define $\Gamma (\m{u})\in (0,\infty)^n$ by
\[
\Gamma(\m{u}) := \left(\,\frac{1}{\m{|u}|}\,
                    \mathbbm{1}_{(\m{0},\m{u}]}(\m{X}_1),
                    \dots, \frac{1}{|\m{u}|}\,
\mathbbm{1}_{(\m{0},\m{u}]}(\m{X}_n)\right),
\]
and define the set  $\Gamma \equiv \{ \Gamma (\m{u}) \mid
\m{u}\in\s\}$. Then $\Gamma$ is a closed and bounded, hence compact,
subset of $[0,\infty)^n$. Thus by  \citet[Theorem~17.2]{MR0274683}
$\overline{\mbox{conv}( \Gamma)} = \mbox{conv}(\overline{\Gamma}) =
\mbox{conv}(\Gamma)$ is also a compact subset of $[0,\infty)^n$. Thus
the continuous function $\prod_{i=1}^n  z_i$ attains its supremum on
$\mbox{conv}(\Gamma)$. Let $S = \mbox{argmax}_{\m{z} \in
\mbox{conv}(\Gamma)} \sum_{i=1}^n \log z_i$. Since the intersection
of $\Gamma$ and the interior $(0,\infty)^n$ of $[0,\infty)^n$ is not
empty, we have $S \subset (0,\infty)^n$.  Since $\sum_{i=1}^n \log
z_i$ is strictly concave,  $S$ consists of a single point,
$\hat{\m{f}} = ( \hat{f}_{i} )_{i=1}^n > \m{0}$. Therefore for any
MLE $\widehat{f}_n$ it follows that the vector $( \widehat{f}_n
(X_i))_{i=1}^n$ is unique. Note that the gradient of $\sum_{i=1}^n
\log z_i$ at $\hat{\m{f}}$ is proportional to $1/\hat{\m{f}} \equiv
(1/\hat{f}_i )_{i=1}^n$.

Now $\mbox{dim}(\mbox{conv}(\Gamma)) = n$; if we consider the $n$
points $\m{u}_i = \m{X_i}$, then the $n$ vectors $\Gamma (\m{u}_i ) =
( 1_{(0,\m{X}_i ]} (\m{X}_1), \ldots , 1_{(0,\m{X}_i ]} (\m{X}_n))/|
\m{X}_i |$, $i=1, \ldots , n$, are almost surely linearly
independent.  (In fact, the matrix $M$ with rows $| \m{X}_i |\Gamma
(\m{X}_i )$, $i=1,\ldots , n$ has $\mbox{det}(M) = 1$ a.s. if the
$\m{X}_i$'s are i.i.d. with any density $f$.)
  By \citet[Theorem 27.4]{MR0274683} the vector $1/\hat{\m{f}}$
belongs to the normal cone of $\mbox{conv}(\Gamma)$ at $\hat{\m{f}}$.
Since $1/\hat{\m{f}} > 0$ we have $\hat{\m{f}} \in \partial (
\mbox{conv}(\Gamma))$ and the plane $\tau$ defined by $\sum_{i=1}^n
z_i /  \hat{f}_i = n$ is a support plane of $\mbox{conv}(\Gamma)$ at
$\hat{\m{f}}$.  Thus for $v_i = 1/(n \hat{f}_i)$, $i=1, \ldots , n$,
it follows that
\begin{eqnarray*}
q(\m{u}) \equiv | \m{u}| - \sum_{i=1}^n v_i 1_{(\m{0}, \m{u}]}
(\m{X}_i ) \ge 0
\end{eqnarray*}
for all $\m{u} \in [0,\infty)^d$ and $q(\m{u} ) = 0$ if $\m{u} =
\m{0}$ or $\Gamma(\m{u}) \in \tau$. We let ${\cal Y} $ denote the set
of vectors $\m{u}$ such that $\Gamma (\m{u}) \in \tau$; i.e. $\Gamma(
{\cal Y} ) = \tau \cap \Gamma $.

The intersection $\tau \cap \mbox{conv}( \Gamma)$ is an exposed face
of $\mbox{conv} (\Gamma)$; see e.g. \citet[p. 162]{MR0274683}. By
\citet[Theorem 18.3]{MR0274683},  $\tau \cap \mbox{conv} (\Gamma) =
\mbox{conv} ( \Gamma ({\cal Y} ))$, and by Theorem 18.1,
$\mbox{supp} (\widehat{G}_n ) \subset {\cal Y}$.  This implies that
for any MLE $\widehat{f}_n$, the support of the corresponding mixing
measure $\widehat{G}_n$ is a subset of ${\cal Y}$, and thus any MLE
has the form (\ref{StructureOfMLEStepOne}) with $\m{y}_j \in {\cal
Y}$ for $j=1, \ldots , m$. To see that $m \le n$, note that $\m{y}_j
\in {\cal Y} \subset A^{\#}$ satisfy
\begin{eqnarray}
| \m{y}_j | = \sum_{i=1}^n v_i 1_{(\m{0}, \m{y}_j ]} ( \m{X}_i ) =
\langle \m{v} , | \m{y}_j | \Gamma (\m{y}_j ) \rangle, \qquad j = 1,
\ldots , m. \label{KeyGradientIdentity}
\end{eqnarray}
Suppose that the vectors  $\{ | \m{y}_j | \Gamma (\m{y}_j )
\}_{j=1}^m$ are linearly dependent; i.e.
\[
 \sum_{j=1}^m b_j | \m{y}_j | \Gamma (\m{y}_j )  = \m{0}
\]
 in $\RR^n$ for some $b_j$, $j=1, \ldots , m$.
Since all the coordinates of the $| \m{y}_j | \Gamma(\m{y}_j ) $
vectors take values in $\{ 0,1 \}$, this system of equations is
algebraically equivalent to the same system in which all the $b_j$'s
take
only integer values, i.e. $b_j \in \ZZ$ for $j=1, \ldots , m$.  \\
\smallskip

\par\noindent
Then it follows on the one hand that
\begin{eqnarray*}
\sum_{j=1}^m b_j \langle \m{v}, | \m{y}_j  | \Gamma (\m{y}_j) \rangle 
& = & \sum_{j=1}^m b_j \sum_{i=1}^n v_i 1_{(\m{0}, \m{y}_j ]} ( \m{X}_i ) \\
& = & \bigg \langle \m{v} , \sum_{j=1}^m b_j | \m{y}_j | \Gamma
          (\m{y}_j ) \bigg \rangle = \langle \m{v} , \m{0} \rangle = 0,
\end{eqnarray*}
and hence, by (\ref{KeyGradientIdentity}), $\sum_{j=1}^m b_j |
\m{y}_j | = 0$, or, since $\m{y}_j = \m{W}_{i_j} \in A^{\#}$ for some
$i_j$,
\[
\sum_{j=1}^m b_j | \m{W}_{i_j}| = 0
\]
 with all $b_j \in \ZZ$.
But this equation has at most countably many solutions $\{|W_{i_j}, j
= 1, \ldots, m\}$, and hence occurs with $P_0^n$-probability $0$.
That is, for any fixed vector $\m{b}= (b_j)_{j=1}^k$ with all $b_j
\in \ZZ$, the function $f_{\m{b}}  ( \m{X}_1 , \ldots , \m{X}_n ) =
\sum_{j=1}^k b_j | \m{W}_{i_j} |$ has at most a finite number of
zeros, so $P_0^n ( f_{\m{b}} ( \m{X}_1 , \ldots , \m{X}_n ) = 0 ) =
0$, and since $\ZZ$ is countable $P_0^n ( \cup_{ \m{b} \in \ZZ^k} \{
f_{\m{b}} ( \m{X}_1 , \ldots , \m{X}_n ) = 0 \} ) = 0$. Thus $P_0^n (
\cap_{\m{b} \in \ZZ^k} \{ f_{\m{b}} ( \m{X}_1 , \ldots , \m{X}_n )
\not = 0 \} ) = 1$. Hence it follows that the linear dependence
condition only holds on an event with probability $0$.

Thus the vectors $| \m{y}_j | \Gamma ( \m{y}_j ) $, $j=1, \ldots , m$
are linearly independent almost surely $P_0^n$, and hence $m \le n$
($P_0^n$ - almost surely). \qedhere
\end{proof}
\smallskip

\begin{lemma}
\label{P:MLEMixingMeasureUnique} The discrete mixing measure
$\widehat{G}_n$ which defines an MLE is $P_0^n-$almost surely unique.
\end{lemma}

\begin{proof}
Suppose that there exist two different MLE's $\widehat{f}_n^1 $ and
$\widehat{f}_n^2$. then
\[
\widehat{f}_n^l (x) = \sum_{j=1}^m \pi_j^l \frac{1}{| \m{y}_j |} 1_{(
\m{0}, \m{y}_j]} (\m{x} ), \ \qquad l = 1,2,
\]
where $\pi_j^l \ge 0$ and $\sum_{j=1}^m \pi_j^l = 1$ for $l = 1,2$.
Therefore
\begin{eqnarray*}
\delta_n (\m{x} ) \equiv \widehat{f}_n^1 (\m{x}) - \widehat{f}_n^2
(\m{x} ) = \sum_{j=1}^m r_j \frac{1}{| \m{y}_j |} 1_{( \m{0},
\m{y}_j]} (\m{x} )
\end{eqnarray*}
where $r_j \equiv \pi_j^1 - \pi_j^2$ has at least $n$ zeros (since we
know that
\[
( \widehat{f}_n^1 (\m{X}_i ))_{i=1}^n = ( \widehat{f}_n^2 (\m{X}_i
))_{i=1}^n = ( \widehat{f}_n (\m{X}_i ))_{i=1}^n
\]
is unique).    So, uniqueness holds if the vectors
\[
(1_{(\m{0}, \m{y}_j ]} (\m{X_i} ))_{i=1}^n \in \{0, 1 \}^n ,  \qquad
\mbox{for} \ \ j = 1, \ldots , m \le n
\]
are (almost surely) linearly independent. But this follows from the
proof of Lemma~\ref{P:MLE-existenceAndStructure}. \qedhere
\end{proof}

Theorem~\ref{T: MLE} does not assert that the MLE is always unique. A
MLE is $P_0^n$ almost surely unique, but we now present an example
          in which there exist an infinite number of MLE's.

        \begin{example}\textbf{[A MLE in $\ffs$ is not always unique]} \label{E: MLE SMU NU}
            To be able to graphically illustrate the set $\Gamma$, in the proof of
            Theorem~\ref{T: MLE}, we need to restrict consideration
            to $n=2$ and in order that we be able to graphically illustrate the MLE(s) we need to
            restrict consideration to $d=2$. Suppose that $\m{X}_1=(1,3)$ and $\m{X}_2=(3,2)$
            are the observation points. The set
            \[
                \Gamma \equiv \left\{ \frac{1}{u_1 u_2}
                \left( \mathbbm{1}_{(\m{0},\m{u}]}(\m{X}_1),
                \mathbbm{1}_{(\m{0}, \m{u}]}(\m{X}_2)\right)\;\bigg|\;
                \m{u}=(u_1, u_2) \in (0,\infty)^2 \right\}
            \]
            and its convex hull,
            $\operatorname{Conv}(\Gamma)$, are illustrated in Figure~\ref{F: MLESMU Ex1}.

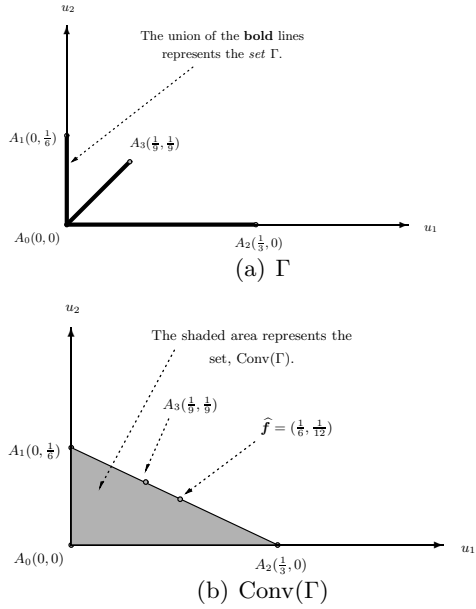
\begin{figure}[htp]
\centering \subfigure[$\Gamma$]
                {
\label{F: MLESMU G1} \scalebox{0.55}{
\setlength{\unitlength}{0.254mm}
\begin{picture}(487,237)(130,-496)
\allinethickness{1.016mm}\path(365,-475)(185,-475) 
\allinethickness{0.254mm}\special{sh 0.3}\put(365,-475){\ellipse{4}{4}} 
\allinethickness{0.254mm}\special{sh 0.3}\put(185,-475){\ellipse{4}{4}} 
\put(345,-496){\shortstack{$A_2(\frac{1}{3},0)$}} 
                                \put(135,-491){\shortstack{$A_0(0,0)$}} 
                                \allinethickness{1.016mm}\path(185,-390)(185,-475) 
                                \allinethickness{0.254mm}\special{sh 0.3}\put(185,-390){\ellipse{4}{4}} 
                                \put(130,-396){\shortstack{$A_1(0,\frac{1}{6})$}} 
                                \allinethickness{1.016mm}\path(245,-415)(185,-475) 
                                \allinethickness{0.254mm}\special{sh 0.3}\put(245,-415){\ellipse{4}{4}} 
                                \put(245,-401){\shortstack{$A_3(\frac{1}{9},\frac{1}{9})$}} 
                                \allinethickness{0.254mm}\put(185,-475){\vector(0,1){190}} 
                                \put(525,-481){\shortstack{$u_1$}} 
                                \put(180,-271){\shortstack{$u_2$}} 
                                \allinethickness{0.254mm}\put(185,-475){\vector(1,0){325}} 
                                \put(255,-316){\shortstack{The union of the \textbf{bold} lines \\\\represents the \emph{set} $\Gamma$.}}
\allinethickness{0.254mm}\dottedline{5}(305,-325)(190,-415)\special{sh
1}\path(190,-415)(195,-413)(194,-412)(193,-411)(190,-415)

\end{picture}
} } \subfigure[$\operatorname{Conv}(\Gamma)$] { \label{F: MLESMU CG1}
\scalebox{0.6}{ \setlength{\unitlength}{0.254mm}
\begin{picture}(448,237)(130,-496)
                                \allinethickness{0.254mm}\path(365,-475)(185,-475) 
                                \allinethickness{0.254mm}\special{sh 0.3}\put(365,-475){\ellipse{4}{4}} 
                                \allinethickness{0.254mm}\special{sh 0.3}\put(185,-475){\ellipse{4}{4}} 
                                \put(345,-496){\shortstack{$A_2(\frac{1}{3},0)$}} 
                                \put(135,-491){\shortstack{$A_0(0,0)$}} 
                                \allinethickness{0.254mm}\path(185,-390)(185,-475) 
                                \allinethickness{0.254mm}\special{sh 0.3}\put(185,-390){\ellipse{4}{4}} 
                                \put(130,-396){\shortstack{$A_1(0,\frac{1}{6})$}} 
                                \put(265,-356){\shortstack{$A_3(\frac{1}{9},\frac{1}{9})$}} 
                                \allinethickness{0.254mm}\put(185,-475){\vector(0,1){190}} 
                                \put(525,-481){\shortstack{$u_1$}} 
                                \put(180,-271){\shortstack{$u_2$}} 
                                \allinethickness{0.254mm}\put(185,-475){\vector(1,0){325}} 
                                \put(255,-316){\shortstack{The shaded area represents the \\\\set, Conv($\Gamma$).}} 
                                \allinethickness{0.254mm}\special{sh 0.3}\path(185,-475)(185,-390)(365,-475)(185,-475) 
                                \allinethickness{0.254mm}\special{sh 0.3}\put(250,-420){\ellipse{4}{4}} 
                                \allinethickness{0.254mm}\special{sh 0.3}\put(280,-435){\ellipse{4}{4}} 
\allinethickness{0.254mm}\dottedline{5}(275,-305)(210,-425)\special{sh 1}\path(210,-425)(213,-420)(212,-420)(211,-420)(210,-425) 
\put(350,-376){\shortstack{$\m{\widehat{f}}=(\frac{1}{6},\frac{1}{12})$}} 
\allinethickness{0.254mm}\dottedline{5}(345,-385)(285,-430)\special{sh 1}\path(285,-430)(290,-428)(289,-427)(288,-426)(285,-430) 
\allinethickness{0.254mm}\dottedline{5}(270,-365)(250,-410)\special{sh 1}\path(250,-410)(253,-405)(252,-405)(251,-405)(250,-410) 
\end{picture}
}
 }
\caption{The sets $\Gamma$ and $\operatorname{Conv}(\Gamma)$ based on
                two observations: $\m{X}_1=(1,3)$ and
                $\m{X}_2=(3,2)$.}
                \label{F: MLESMU Ex1}
\end{figure}

Using \citet[Theorem~22, pg. 118]{Lin}, it follows that any MLE,
            $\widehat{f}_2$, will have a unique value for  $\m{\widehat{f}}\equiv (\widehat{f}_2(\m{X}_1),
            \hat{f}_2(\m{X}_2))$ that is given by $\m{\widehat{f}} = (\tilde{w}_1^{-1}, \tilde{w}_2^{-1})$
            where $\m{\tilde{w}} = (\tilde{w}_1, \tilde{w}_2)$ maximizes the function
            $(w_1,w_2) \mapsto \log(w_1 w_2)$ on the set
            \[
                \left\{ (w_1, w_2) \in (0,\infty)^2\;\bigg|\; \frac{w_1}{3}\leq 2\quad\text{and}\quad \frac{w_2}{6}\leq 2\right\}\,.
            \]
            It is immediate that $\m{\tilde{w}} = (6, 12)$ from which we conclude that
            $\m{\tilde{f}} = (1/6, 1/12)$ has exactly two representations as a convex combination
            of extreme elements in $\operatorname{Conv}(\Gamma)$ (see Figure~\ref{F: MLESMU CG1} again):
            \begin{align*}
                & && \left(\frac{1}{6}, \frac{1}{12}\right) = \frac{1}{2}
                \left(0, \frac{1}{6}\right) + \frac{1}{2}\left(\frac{1}{3}, 0\right)\,,\\
                &\text{and }&&
                \left(\frac{1}{6}, \frac{1}{12}\right) = \frac{1}{4}
                \left(\frac{1}{3}, 0\right) + \frac{3}{4}\left(\frac{1}{9}, \frac{1}{9}\right)\,.
            \end{align*}
            These two convex combinations yield two different maximum likelihood estimators, as shown in Figures
            \ref{F: MLESMU MLE1.1} and \ref{F: MLESMU MLE1.2}.

            \begin{figure}[htp]
                \centering
                \subfigure[Example~\ref{E: MLE SMU NU} : MLE 1]
                {
                    \label{F: MLESMU MLE1.1}
                    \scalebox{0.7}{
                        \setlength{\unitlength}{0.254mm}
                        \begin{picture}(334,282)(85,-426)
                                \allinethickness{0.254mm}\put(105,-405){\vector(1,0){240}} 
                                \allinethickness{0.254mm}\put(105,-405){\vector(0,1){240}} 
                                \put(90,-416){\shortstack{$0$}} 
                                \put(160,-426){\shortstack{$1$}} 
                                \put(220,-426){\shortstack{$2$}} 
                                \put(280,-426){\shortstack{$3$}} 
                                \put(355,-411){\shortstack{$\theta_1$}} 
                                \put(100,-156){\shortstack{$\theta_2$}} 
                                \put(85,-351){\shortstack{$1$}} 
                                \put(85,-291){\shortstack{$2$}} 
                                \put(85,-231){\shortstack{$3$}} 
                                \allinethickness{0.254mm}\special{sh 0.5}\path(105,-285)(165,-285)(165,-405)(105,-405)(105,-285) 
                                \allinethickness{0.254mm}\special{sh 0.35}\path(105,-225)(165,-225)(165,-285)(105,-285)(105,-225) 
                                \allinethickness{0.254mm}\special{sh 0.2}\path(165,-285)(285,-285)(285,-405)(165,-405)(165,-285) 
                                \put(130,-346){\textbf{\shortstack{$\frac{1}{4}$}}} 
                                \put(130,-261){\textbf{\shortstack{$\frac{1}{6}$}}} 
                                \put(215,-346){\textbf{\shortstack{$\frac{1}{12}$}}} 
                                \allinethickness{0.254mm}\special{sh 0.55}\put(165,-225){\ellipse{4}{4}} 
                                \put(160,-211){\textbf{\shortstack{$\m{X}_1$}}} 
                                \put(280,-271){\textbf{\shortstack{$\m{X}_2$}}} 
                                \allinethickness{0.254mm}\special{sh 0.55}\put(285,-285){\ellipse{4}{4}} 
                        \end{picture}
                    }
                }
                \subfigure[Example~\ref{E: MLE SMU NU} : MLE 2]
                {
                    \label{F: MLESMU MLE1.2}
                    \scalebox{0.7}{
                        \setlength{\unitlength}{0.254mm}
                        \begin{picture}(323,282)(85,-426)
                                \allinethickness{0.254mm}\put(105,-405){\vector(1,0){240}} 
                                \allinethickness{0.254mm}\put(105,-405){\vector(0,1){240}} 
                                \put(90,-416){\shortstack{$0$}} 
                                \put(160,-426){\shortstack{$1$}} 
                                \put(220,-426){\shortstack{$2$}} 
                                \put(280,-426){\shortstack{$3$}} 
                                \put(355,-411){\shortstack{$\theta_1$}} 
                                \put(100,-156){\shortstack{$\theta_2$}} 
                                \put(85,-351){\shortstack{$1$}} 
                                \put(85,-291){\shortstack{$2$}} 
                                \put(85,-231){\shortstack{$3$}} 
                                \allinethickness{0.254mm}\special{sh 0.35}\path(105,-225)(165,-225)(165,-405)(105,-405)(105,-225) 
                                \allinethickness{0.254mm}\special{sh 0.2}\path(165,-225)(285,-225)(285,-405)(165,-405)(165,-225) 
                                \allinethickness{0.254mm}\dottedline{5}(105,-290)(285,-290) 
                                \put(130,-321){\textbf{\shortstack{$\frac{1}{6}$}}} 
                                \put(215,-321){\textbf{\shortstack{$\frac{1}{12}$}}} 
                                \allinethickness{0.254mm}\special{sh 0.6}\put(165,-225){\ellipse{4}{4}} 
                                \allinethickness{0.254mm}\special{sh 0.6}\put(285,-290){\ellipse{4}{4}} 
                                \put(160,-211){\shortstack{$\m{X}_1$}} 
                                \put(300,-291){\shortstack{$\m{X}_2$}} 
                                \allinethickness{0.254mm}\put(285,-225){\ellipse{4}{4}} 
                                \put(130,-171){\shortstack{Point of support not in $\{\m{X}_1,\m{X}_2\}$.}} 
                                \allinethickness{0.254mm}\dottedline{5}(220,-180)(275,-215)\special{sh 1}\path(275,-215)(269,-213)(270,-212)(271,-211)(275,-215) 
                        \end{picture}
                    }
                }
                \caption[Example~\ref{E: MLE SMU NU} : MLE]{Two maximum likelihood estimators in
                $\ffs(2)$, supported on the grid generated by the data: $\m{X}_1=(1,3)$ and
                $\m{X}_2=(3,2)$. The two figures show the \emph{contour/level plots}
                of the respective maximum likelihood densities.}
                \label{F: MLESMU MLE1}
\end{figure}
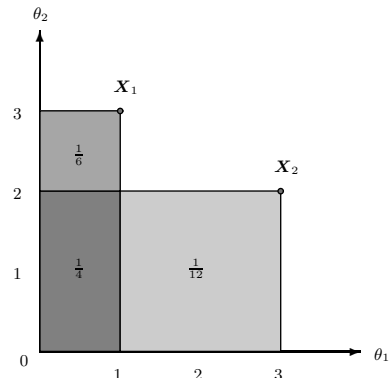
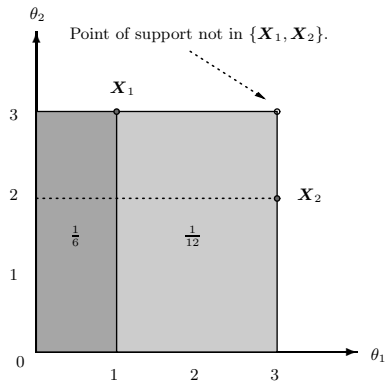

 It should be noted however that infinitely many maximum likelihood estimators exist in this
            case: Observe that the hyperplane that passes through $\m{\widehat{f}}$ intersects
            $\operatorname{Conv}(\Gamma)$
            on the line segment joining the points $(0,1/6)$ and $(1/3, 0)$.
            Then $\m{\widehat{f}}$ can be written
            in infinitely many ways as a convex combination of points on this line segment.
            However, the corresponding MLEs will no longer be supported solely on the grid generated by the
            data.\hfill\qedsymbol
\end{example}

    \subsection{Strong pointwise consistency of the MLE}\label{SS: consistency}

        Let $\m{X}_1,\m{X}_2,\ldots , \m{X}_n , \ldots $
        be the coordinate random elements on the (completed)
        infinite product space
        $(\Omega^{\infty},\mathcal{A}^{\infty},P^{\infty})$ such that
        these coordinates are i.i.d. according to $f_0\equiv f_{G_0}$ on
        $\s$. Let $A\in \mathcal{A}^{\infty}$ be the event (with
        $P^{\infty}$-probability one) that for each $n\in\NN$ there
        exists a unique SMU density, $\widehat{f}_n\equiv f_{\hat{G}_n}$, maximizing the
        log-likelihood.

        From Theorem~\ref{T:InversionFormulaSMU} we have that for each $n\in\NN$
        and a fixed $\omega\in A$, there exists a unique Borel
        probability measure, $\widehat{G}_{n}$ on $(\s,\|\cdot\|_2)$, such
        that
\begin{eqnarray}
\widehat{f}_n(\m{x})
 &=& \int_{\s} \frac{1}{|\m{u}|}\,
                \mathbbm{1}_{(\m{0},\m{u}]}(\m{x})\;\ud
                \widehat{G}_{n}(\m{u}) \notag \\
&=& \int_{\m{u}\ge \m{x}}\,
                \frac{1}{|\m{u}|}\,\ud \widehat{G}_{n}(\m{u})\,.
                \label{E: CS1}
\end{eqnarray}
holds true for \emph{all} $\m{x}\in\s$. We are ready to formulate and
prove the following proposition.

\begin{proposition}\textbf{\upshape{[Strong Consistency of the MLE in $\ffs$]}}
\label{P: Consistency SMU}\hfill
\begin{enumeratea}
\item
    \begin{enumeratei}
        \item The sequence of maximum likelihood mixing
                distributions $\{\widehat{G}_n\}_{n=1}^{\infty}$ converges
                \emph{weakly} to $G_0$ as $n\to\infty$,
                $P^{\infty}$-almost surely.

        \item In addition, for Lebesgue almost all $\m{x}\in\s$,
                $\widehat{f}_{n}(\m{x})\rightarrow_{a.s.}
                f_{0}(\m{x})$ as $n\to\infty$.
                In particular, if $f_0$ is continuous at $\m{x}\in\s$, then
                \[
                    \left| \widehat{f}_{n}(\m{x}) - f_0(\m{x}))\right |
                   \rightarrow_{a.s.} 0 \quad
                    \text{as  $n\to\infty$}.
                \]
    \end{enumeratei}
\item
The sequence of maximum likelihood estimators,
            $\{\widehat{f}_{n}\}_{n=1}^{\infty}$, is strongly consistent
            in the total variation (or $L_1$) and
            in the Hellinger metrics. That is,        \[
                    \int_{\s}\left|\hat{f}_{n}(\m{x}) -
                    f_0(\m{x})\right|\,\ud\m{x}
                   \rightarrow_{a.s.} 0 \quad\text{as
                    $n\to\infty$}\,,
                \]
            and, with $h^2(p,q) = (1/2) \int \{ \sqrt{p(\m{x} )} - \sqrt{q(\m{x})} \}^2 d \m{x} $,
                \[
                    h\left(\widehat{f}_{n},f_0\right)
                   \rightarrow_{a.s.} 0 \quad\text{as
                    $n\to\infty$}\,.
                \]
\end{enumeratea}
\end{proposition}

\begin{proof}
(a) \ (i) \
           To be able to apply Theorems 3.4, 3.5 and 3.7 of \citet{MR944202},
            with the refinement on page 143 of the same article, we need to
            provide the relevant setup as well as establish the
            assumptions of Pfanzagl's theorems. We do this below.
            \bigskip

            Let $\mathcal{C}_{0}\left(\s,\|\cdot\|_2\right)$ denote the set of all
            real-valued, continuous functions on $\s$ that vanish at
            $\infty$. Let $\Theta_{\ast}$ denote the set of all Borel
            sub-probability measures on $\s$, equipped with the vague
            topology, $\tau$, which makes the space a compact, metrizable,
            topological space -- and thus with a countable base.
            It is also a convex subset of the linear space of all
            finite, signed, Borel measures on $(\s,\|\cdot\|_2)$.
            For clarity, the vague topology is the smallest topology that
            makes the functions
            \[
                \mu \mapsto \int_{\s} g(\m{x})\;\ud\mu(\m{x})
            \]
            continuous, for each $g \in \mathcal{C}_{0}\left(\s,\|\cdot\|_2\right)$.
            By metrizability, the topology $\tau$ is completely
            characterized by convergent sequences, $\theta_n
            \stackrel{v}{\Rightarrow} \theta$ as $n\to\infty$, on
            $(\Theta_{\ast},\tau)$.

            Let also $\Theta \subseteq \Theta_{\ast}$ be the set of all
            Borel probability measures on $\s$, and notice that $\mu \in
            \Theta$. Also, for each $\theta_{\ast}\in\Theta_{\ast}$
            there exists a unique $c\in [0,1]$ and a unique
            $\theta\in\Theta$, such that $\theta_{\ast}=c\theta$.
            Further, notice that letting $m(\nu,\cdot)\equiv
            f_{\nu}(\cdot)$, for each $\nu\in\Theta_{\ast}$,
            and $M_n(\cdot)\equiv \PP_n \log\left\{
            m(\cdot,\m{X})\right\}$, we have
            \[
                M_n(\theta_{\ast}) = \log\{c\} + M_n(\theta) \leq
                M_n(\theta),\quad\text{since $c\in[0,1]$},
            \]
            whence, $\sup_{\theta\in\Theta_{\ast}}
            \left(M_n(\theta)\right) = \sup_{\theta\in\Theta}
            \left(M_n(\theta)\right)$.

            With reference measure the Lebesgue measure $\lambda\equiv Q$ and
            for each $\nu\in\Theta_{\ast}$, let $P_{\nu}\in\Theta_{\ast}$ be
            the sub-probability, Borel measure on $(\s,\|\cdot\|_2)$ with
            Radon-Nikodym derivative with respect to $\lambda$ being
            $f_{\nu}$, Lebesgue almost surely. Then by virtue of
            Fubini-Tonelli, $P_{\nu}\in\Theta$ when and only when
            $\nu\in\Theta$. Also, notice that for each fixed $\m{x}\in\s$,
            the functional $\nu \mapsto f_{\nu}(\m{x})$ is
            not vaguely continuous
            at any $\nu\in\Theta_{\ast}$ with a discontinuity point
            on the boundary of $[\m{x},\infty)$.
            However, since for a fixed $\m{x}\in\s$, the function
            $\m{y} \mapsto \mathbbm{1}_{[\m{x},\infty)}(\m{y})/{|\m{y}|}$ is
            easily seen to be an upper semi-continuous function on
            $\s$ -- vanishing at $\infty$,
            \citet{MR1253752}, Theorem 10, p. 138, applies and asserts that the function $\nu \mapsto
            f_{\nu}(\m{x})$ on $(\Theta_{\ast},\tau)$ is itself (vaguely)
            upper semi-continuous. Since this holds for all $\m{x}\in\s$,
            it holds almost-surely. Also, the mapping
            $\nu \mapsto f_{\nu}(\m{x})$ is affine on $\Theta_{\ast}$ (and
            hence concave also.)

            It remains to establish that for each fixed $\tau$-open subset
            $U$ of $\Theta_{\ast}$, the real-valued function $T_{U}(\cdot)$
            on $\s$ defined by
            \[
                T_{U}(\m{x}) = \sup_{\nu\in U}\left\{\int_{(0,\infty)^d}
                \frac{1}{|\m{u}|}\,
                \boldsymbol{1}_{(\m{0},\m{u}]}(\m{x})\,
                \mathrm{d}\nu(\m{u})\right\}
            \]
            is a ${\cal A}$-measurable function. We can choose to take ${\cal A}$
            to be the Lebesgue $\sigma$-field, in which case measurability follows
            by observing that $T_U(\cdot)$ is a block-decreasing function
            and appeal to Proposition~\ref{P: LMBDF}.

            We now apply our setup to Theorem 3.4 of \citet{MR944202} and further
            appeal to the fact that a vaguely convergent sequence of probability
            measures with limit a probability measure, is, in fact, weakly
            convergent. This gives the desired conclusion:
            the random sequence of maximum likelihood mixing
            probability measures $\{\hat{G}_n\}_{n=1}^{\infty}$ converges
            weakly to $G_0$ as $n\to\infty$,
            $P^{\infty}$-almost surely.\\
\smallskip

\par\noindent
(ii) \ Combining the fact that, for each fixed $\m{x}\in\s$,
            $\nu \mapsto f_{\nu}(\m{x})$ is vaguely upper semi-continuous
            on $\Theta_{\ast}$ with the conclusion of part (a)(i),
            we get    \begin{equation}\label{E:uppers1}
                \varlimsup_{n\to\infty}\left\{f_{\widehat{G}_n}(\m{x})\right\}
                \leq f_{0}(\m{x});\; P^{\infty}\text{-a.s.}
                \quad\text{for all $\m{x}\in\s$}.
            \end{equation}

Let
\[
F_{G_0}(\cdot)= \int_{\s}
            \frac{\left|\cdot\wedge\m{u}\right|}{\left|\m{u}\right|}
            \,\ud G_0(\m{u})
 \]
and
\[
F_{\widehat{G}_n}(\cdot)=\int_{\s}
            \frac{\left|\cdot\wedge\m{u}\right|}{\left|\m{u}\right|}
            \,\ud\widehat{G}_n(\m{u})
\]
be the distribution functions corresponding to the
 densities $f_{0}(\cdot)$ and
$\widehat{f}_{n}(\cdot)$, respectively, $n\in\NN$. These distribution
            functions are everywhere
            continuous on the Euclidean set $\s$. In fact, since for each
            fixed $\m{x}\in\s$, the function $\m{u} \mapsto
            \left|\m{x}\wedge\m{u}\right|/\left|\m{u}\right|$ is
            bounded (by $1$) and continuous on $\s$,we then have that
            \begin{equation}
                F_{\widehat{G}_n}(\m{x})\to_{a.s.} F_{G_0}(\m{x})
                \;
                \text{for all
                $\m{x}\in\s$}\label{E:emp1}
            \end{equation}
            follows directly by the definition of almost sure weak convergence
            of the mixing random measures $\{\widehat{G}_n\}_{n=1}^{\infty}$ to
            $G_0$, established in part (a)(i).

            Let $B$ be the set of
            points on $\s$ at which $f_{0}$ is continuous.
            Then  $B^c$ has
            Lebesgue measure zero, $\lambda(B^c)=0$, exactly because
            $f_0$ is discontinuous on the boundary $\partial
            [\m{x}_0,\infty)$ for a (possibly non-existent) $\m{x}_0\in\s$
            where $P_0$ is discontinuous (i.e. such that $P_0(\{\m{x}_0\})>0$.)
            Since $P_0$ can have at most countably many discontinuity points
            $\m{x}_0\in\s$ and since $\lambda(\partial [\m{x}_0,\infty))=0$,
            we get by countable subadditivity of $\lambda$ that indeed
            $\lambda(B^c)=0$.

            Fix arbitrary $\m{x}\in B$ and $\epsilon >0$. Then, since $f_0$ is
            lower semi-continuous at $\m{x}$, there exists an open
            neighborhood $U_{\m{x},\epsilon}$ of $\m{x}$ such that for every
            $\m{y}\in U_{\m{x},\epsilon}$ we have that $f_0(\m{y}) > f_0(\m{x})
            - \epsilon$. In particular, there exists an
            $U_{\m{x},\epsilon}\ni\m{x}_{\epsilon}> \m{x}$ satisfying
            $f_0(\m{x}_{\epsilon}) > f_0(\m{x}) - \epsilon$. Since $f_0$ is
            block-decreasing, we have:
            \begin{equation}\label{E:uppers2}
                \frac{V_{F_{G_0}}
                \left(\m{x},\m{x}_{\epsilon}\right]}
                {\lambda\left((\m{x},\m{x}_{\epsilon}]\right)}
                = \frac{\int_{(\m{x},\m{x}_{\epsilon}]}\left\{
                f_{0}(\m{y})\right\}\;\ud\m{y}}
                {\lambda\left((\m{x},\m{x}_{\epsilon}]\right)}\geq
                f_0(\m{x}_{\epsilon}) >
                f_{0}(\m{x}) - \epsilon\;.
            \end{equation}

Further, for each fixed $n\in\NN$, since $\widehat{f}_{n}(\cdot)$ is
block-decreasing (as a SMU density), we have
\begin{eqnarray}
f_{\widehat{G}_n}(\m{x}) &\geq&
\frac{\int_{(\m{x},\m{x}_{\epsilon}]}\left\{
                f_{\widehat{G}_n}(\m{y})\right\}\;\ud\m{y}}
                {\lambda\left((\m{x},\m{x}_{\epsilon}]\right)}
                \label{E:uppers3}\\
&=& \frac{V_{F_{\widehat{G}_n}}
                \left(\m{x},\m{x}_{\epsilon}\right]}
                {\lambda\left((\m{x},\m{x}_{\epsilon}]\right)}\;.
                \label{E:uppers4}
\end{eqnarray}
Equation~\eqref{E:emp1} further implies that
\begin{equation}\label{E:emp2}
                V_{F_{\widehat{G}_n}}
                \left(\m{x},\m{x}_{\epsilon}\right] \to
                V_{F_{G_0}}\left(\m{x},
                \m{x}_{\epsilon}\right],\quad\text{as $n\to\infty$}.
\end{equation}
Combining equations \eqref{E:uppers2}--\eqref{E:emp2} and the fact
that $\epsilon > 0$ was arbitrary, we get
\begin{equation}\label{E:uppers3}
                \varliminf_{n\to\infty}\left\{f_{\widehat{G}_n}(\m{x})\right\}
                \geq f_{0}(\m{x});\; P^{\infty}\text{-a.s.}
                \quad\text{for $\m{x}\in B$}.
\end{equation}
Equations \eqref{E:uppers1} and \eqref{E:uppers3} yield the
            assertion:
for Lebesgue almost all $\m{x}\in\s$
            (and, in particular, at the points of continuity of $f$),
                $f_{\widehat{G}_n}(\m{x})\rightarrow_{a.s.}
                f_{0}(\m{x})$ as $n\to\infty$ holds.\\
\smallskip
\par\noindent
(b)  Showing consistency in the $L_1$ (total-variation) norm
            is a direct consequence of part (a) (ii) and
            Glick's Theorem, \citet{MR0375606});  see also
            \citet{MR891874}, p. 25.

            Convergence in the Hellinger metric follows from
            the following well-known
            inequalities of \citet[p.46]{MR856411}:
            \[
                h^2(P,Q)\leq \frac{1}{2} \|P-Q\|_{L_1} \leq h(P,Q)
                {\left\{ 2 - h^2(P,Q)\right\}}^{\frac{1}{2}},
            \]
            where $h^2(P,Q)=2^{-1}\int\left(\sqrt{\ud P} -
            \sqrt{\ud Q}\,\right)^2$ is the squared Hellinger metric and
            $\|\cdot\|_{L_1}$ is the $L_1$-norm.

            \qedhere
    \end{proof}

    \section{A local asymptotic minimax lower bound} \label{S: minimax}

Let $\m{X}_i := (X_{i,1},\dots,X_{i,d})'$ for
        $i=1,2,\dots,n$ be i.i.d. random vectors from
        density $f \in \ffs(d)$.
        For a fixed $\m{x}_0 \equiv
        (x_{0,1},\dots,x_{0,d})'\in\s$, we want to estimate the
        functional $T(f) := f(\m{x}_0)$ on the basis of
        $\m{X}_1,\dots,\m{X}_n$. We shall make the following
        assumption:
        \begin{assumption}
        \label{A: LMLB2}
            Suppose that
            $f\in\ffs$ is continuously differentiable at $\m{x}_0$,\\
            $f(\m{x}_0)>0$, and, in particular, there exists an open ball
            $A(\m{x}_0)$ around $\m{x}_0$ such that $f$ is everywhere
            strictly positive on $A(\m{x}_0)$ and where
            $(\partial/{\partial x_j})f(\m{x}_0) < 0$
            exist for all $j \in \{1,2,\dots,d\}$ and are
            continuous on $A(\m{x}_0)\subseteq\s$.
            Further, we assume that the
            full
            mixed derivative of $f$ exists,  is continuous
            on $A(\m{x}_0)$, and satisfies
            \[
                (-1)^d\left.\frac{\partial^d f}{\partial x_1 \cdots
                \partial x_d}(\m{x})\right|_{\m{x}=\m{y}} > 0 \qquad
                \mbox{for all} \ \m{y}\in A(\m{x}_0) .
            \]
        \end{assumption}

 \begin{proposition}\label{P: LAMLB2}
                Suppose that
                $f \in {\cal F}_{SMU}$
                satisfies
                Assumption~\ref{A: LMLB2} at the fixed point $\m{x}_0 \in\s.$ Then there is a sequence 
                $\{ f_n \} \subset {\cal F}_{SMU} $ such that any estimator sequence $\{ T_n \}$ of $f(x_0)$ satisfies
                \begin{gather}
                    \varliminf_{n\to\infty}
                    \left\{
                    \E_{f_n}\left\{n^{\frac{1}{3}}\left|T_n-
                    f_n(\m{x}_0)\right|\right\} ,
                    \E_{f}\left\{n^{\frac{1}{3}}\left|T_n -
                    f(\m{x}_0)\right|\right\}\right\} \notag \\
                    \geq \frac{e^{-\frac{1}{3}}}{2^{d}}{\left\{
                    3^{d-1}\right\}}^{\frac{1}{3}}{\left\{(-1)^d
                    \left.\frac{\partial^{d} f(\m{x})}{\partial x_1 \cdots \partial
                    x_d}\right|_{\m{x}=\m{x}_0} \cdot f(\m{x}_0)
                    \right\}}^{\frac{1}{3}}\,. \label{E: LM SMU}
                \end{gather}
        \end{proposition}

        \begin{remark}
        The lower bound in Proposition~\ref{P: LAMLB2} should be contrasted to a similar
        lower bound for estimation of $f(\m{x}_0) $ for $f \in {\cal F}_{BDD}$ which is derived by
        \citet{Pavlides:09}.  In that case the natural hypothesis is $\partial f (\m{x}_0 )/ \partial x_i < 0$
        for $i = 1, \ldots , d$,
        and the resulting rate of convergence is $n^{1/(d+2)}$.
        \end{remark}

        To prove Proposition~\ref{P: LAMLB2}  we will make use of the following
        lemma. It
        was established in the form presented here by \citet{MR1370294};
        see also \citet{MR1600884} and \citet{MR1792307}.

        \begin{lemma}\label{P: GJ}
            Let $\ff$ be a class of densities on a measurable space
            $(\mathcal{X},\mathcal{A})$ and $f$ a \emph{fixed} element of
            $\ff$. Let $\ff_{f}$ denote any open Hellinger ball with center
            $f\in\ff$. Assume that there exists a sequence
            ${\{f_n\}}_{n=1}^{\infty}\subseteq\ff$ such
            that
            \begin{equation}\label{E: GJ1}
                \lim_{n\to\infty} \left\{ \sqrt{n} h(f_n,f) \right\}
                = \alpha
            \end{equation}
            and
            \begin{equation}\label{E: GJ2}
                \lim_{n\to\infty} \left| T(f_n) - T(f) \right| = \beta
            \end{equation}
            both hold for some constants $0< \alpha,\beta <\infty$, and
            where $T$ is a \emph{functional} on $\ff$. Here, $h^2(f_n,f)\equiv 2^{-1}\int
            \{\sqrt{f_n(x)}-\sqrt{f(x)}\}^2 \,\ud\mu(x)$, is the Hellinger
            distance between the $\mu$-densities $f_n$ and $f$.
            Let $l ( \cdot ) $ be a convex function, symmetric about zero,
            which is non-decreasing on $[0,\infty)$.

            Then, it holds that
            \begin{equation}\label{E: GJ4}
                \varliminf_{n\to\infty} \left\{
                R_{n,l}(\ff_{f})\right\} \geq l\left( \frac{1}{4} \beta e^{-2
                {\alpha}^2}\right)
            \end{equation}
            where $R_{n,l}(\ff)\equiv
            \inf_{T_n}\sup_{g\in\ff}\E_{g^{\otimes n}}\{l(T_n -T(g))\}$
            is the minimax risk
            for estimating the functional $T(f)$ based on $n$ i.i.d
            observations from $\ff$.

            In particular, for the loss $l(x) = |x|$ on
             we have
             \begin{equation}\label{E: GJ3}
                \varliminf_{n\to\infty} \left\{
                R_{n,|\cdot|}(\ff_{f})\right\} \geq \frac{1}{4} \beta e^{-2
                {\alpha}^2}\,.
            \end{equation}
        \end{lemma}
        \bigskip

        Hereafter, fix an otherwise arbitrary vector $\m{h} :=
        (h_1,\dots,h_d)\in\s$, and define $\m{H} :=
        \operatorname{diag}(\m{h})\in M_{d\times d}\left((0,\infty)\right).$
        For each $k\in\NN$, consider the perturbation rectangle
        \[
            I_{n}(k) := \bigotimes_{i=1}^{d} \left[x_{0,i}-
            n^{-\frac{1}{k}}h_i, x_{0,i}+
            n^{-\frac{1}{k}}h_i \right]\,,
        \]
        only for those positive integers $n\geq n_0(k,\m{x}_0,\m{h})$ for
        which $I_n(k)\subseteq A(\m{x}_0)$ for all $n\geq n_0$. The
        two-dimensional case, $d=2$, is illustrated in Figure~\ref{F: Song Rectangle}.

        \begin{figure}[htp]
                \centering
                \scalebox{0.7}{
                    \setlength{\unitlength}{0.254mm}
                    \begin{picture}(521,357)(20,-536)
                            \allinethickness{0.254mm}\put(120,-520){\vector(0,1){315}} 
                            \allinethickness{0.254mm}\put(120,-520){\vector(1,0){365}} 
                            \allinethickness{0.254mm}\put(312,-372){\ellipse{245}{245}} 
                            \allinethickness{0.254mm}\special{sh 0.2}\path(235,-295)(390,-295)(390,-455)(235,-455)(235,-295) 
                            \allinethickness{0.254mm}\special{sh 0.3}\put(310,-375){\ellipse{4}{4}} 
                            \allinethickness{0.254mm}\dottedline{5}(235,-455)(235,-520) 
                            \allinethickness{0.254mm}\dottedline{5}(390,-455)(390,-520) 
                            \allinethickness{0.254mm}\dottedline{5}(310,-375)(310,-520) 
                            \allinethickness{0.254mm}\dottedline{5}(310,-375)(120,-375) 
                            \allinethickness{0.254mm}\dottedline{5}(235,-295)(120,-295) 
                            \allinethickness{0.254mm}\dottedline{5}(235,-455)(120,-455) 
                            \put(95,-376){\shortstack{$x_{02}$}} 
                            \put(300,-536){\shortstack{$x_{01}$}} 
                            \put(185,-536){\shortstack{$x_{01} - n^{-1/k}h_1$}} 
                            \put(350,-536){\shortstack{$x_{01} + n^{-1/k}h_1$}} 
                            \put(20,-461){\shortstack{$x_{02} - n^{-1/k}h_2$}} 
                            \put(20,-301){\shortstack{$x_{02} + n^{-1/k}h_2$}} 
                            \put(445,-321){\shortstack{$A((x_{01},x_{02}))$}} 
                            \allinethickness{0.254mm}\put(415,-405){\vector(1,2){40}} 
                            \allinethickness{0.254mm}\put(265,-230){\vector(1,0){105}} 
                            \put(385,-236){\shortstack{$I_n(k)$}} 
                            \allinethickness{0.254mm}\path(265,-230)(265,-320) 
                            \put(500,-526){\shortstack{$x$}} 
                            \put(115,-191){\shortstack{$y$}} 
                            \put(100,-536){\shortstack{$(0,0)$}} 
                    \end{picture}
                }
                \caption{Perturbation rectangle $I_n(k)$, for the case $d=2$,
                with center $\m{x}_0=(x_{01},x_{02})$ and $\m{h}=(h_1,h_2)$.}
                \label{F: Song Rectangle}
        \end{figure}
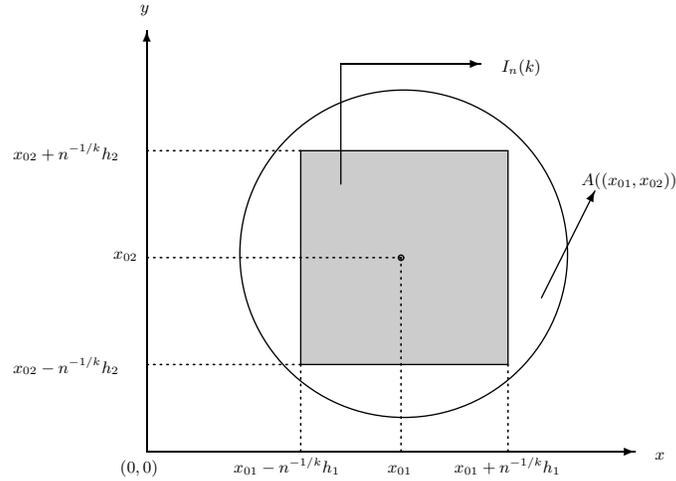

        Recall Assumption~\ref{A: LMLB2}. Let $b := \left.
        ({\partial^d}/{\partial x_1 \cdots \partial x_d})f(\m{x})
        \right|_{\m{x}=\m{x}_0}$ and observe that $(-1)^d b > 0.$
        Finally, define the functions $h_n$ on $I_n (3d)$ as follows:
        \begin{align*}
            h_n(y_1,\dots,y_d)\quad &:= (-1)^d \prod_{i=1}^{d} \left\{
            \mathbbm{1}_{\left(x_{0,i},
            x_{0,i}+n^{-\frac{1}{3d}}h_i\right]}(y_i)
            - \mathbbm{1}_{\left[x_{0,i}-n^{-\frac{1}{3d}}h_i,
            x_{0,i}\right]}(y_i)
            \right\}\,, \\
            \intertext{and}
            g_n(\m{y})\quad &:= b\int_{\m{u} \succeq \m{y}} \left\{
            \mathbbm{1}_{I_n(3d)}(\m{u}) \cdot h_n(\m{u})\right\}\,
            \ud\m{u},
        \end{align*}
        where we  observe that $g_n(\m{y})\geq 0$
        for all $\m{y}\in I_n(3d)$, since $\m{x}_0$ is the
        center of the rectangle $I_n(3d)$. In fact,
        consideration of the geometry of the definition of
        $g_n(\cdot)$ reveals that, for $\m{y}\in I_n$,
        $g_n(\m{y})$ is equal to $(-1)^db>0$ times the volume of the
        rectangle $[\m{v}_n(\m{y})\wedge\m{y}, \m{v}_n(\m{y})\vee\m{y}]$, where
        $\m{v}_n(\m{y})$ is defined as that vertex of $I_n$
        that is closest in $L_2$-distance from $\m{y}\in I_n$.
        Since $I_n$ is a decreasing sequence of compact sets, it
        is then immediately clear that $g_n(\m{y})$ is (pointwise)
        non-increasing in $n\in\NN$, for each fixed $\m{y}\in\s$.

        Assume that $f \in \ffs$, and for fixed vectors
        $\m{x}_0, \m{h} \in\s$ we further assume that $f$ satisfies
        Assumption~\ref{A: LMLB2}. For $n \geq
        n_0(3d,\m{x}_0,\m{h})$, define the perturbed density, $f_n$
        of $f$ at $\m{x}_0$, by
        \begin{equation}
            f_n(\m{x})\quad =
            \begin{cases}
                \dfrac{f(\m{x}) + \theta g_n(\m{x})}{d_n} &:\quad \text{if } \m{x}\in
                I_n(3d) \\
                \dfrac{f(\m{x})}{d_n} &:\quad \text{if }
                \m{x}\in I_n^{c}(3d)
            \end{cases}
        \end{equation}
        for some arbitrary but fixed $\theta\in (0,1)$ and where $d_n$ is the
        normalizing constant for $f_n$, uniquely determined by
        $\int_{\s} f_n(\m{x})\,\ud\m{x} = 1.$ We will see the
        importance of the value of $b$ and the fact that $0<\theta<1$
        in the following proposition that establishes that
        $\{f_n\}_{n\geq n_1}\subseteq\ffs(d)$ for a sufficiently large
        $n_1\in\NN$.

        \begin{proposition}\label{P: closureSMU}
            There exists a positive integer $n_1 := n_1(d,\m{x}_0,\m{h})
            \geq n_0(3d,\m{x}_0,\m{h})$ such that $f_n \in \ffs$ for all
            $n \geq n_1$.
        \end{proposition}

        \begin{proof}
            Since $f\in\ffs(d)$, we get from
            Theorem~\ref{T:InversionFormulaSMU} that
            \begin{equation}\label{E: cl1}
                V_f[\m{x},\m{y}]\geq 0\,,\quad
                \text{for all $d$-boxes $[\m{x},\m{y}]$.}
            \end{equation}
            From the definition of $g_n(\cdot)$, we see that its
            full, mixed partial derivative exists
            in a neighborhood of $\m{x}_0$.
            Hence, by definition and the fact
            that $(-1)^db>0$ and $\theta\in(0,1)$,
            we have that
      \begin{eqnarray}
            \lefteqn{(-1)^d \left. \frac{\partial^d f_n}
                    {\partial x_1 \cdots \partial
                    x_d}(\m{x})\right|_{\m{x}=\m{y}}
                     \geq\quad (-1)^d \left. \frac{\partial^d f}
                    {\partial x_1 \cdots \partial
                    x_d}(\m{x})\right|_{\m{x}=\m{y}} -
                    (-1)^db\theta} \nonumber \\
             & =&  \left[(-1)^d \left. \frac{\partial^d f}
                    {\partial x_1 \cdots \partial
                    x_d}(\m{x})\right|_{\m{x}=\m{y}} - (-1)^db\right]
                    + (1-\theta)(-1)^db\nonumber \\
            & \geq & 2^{-1}(1-\theta)(-1)^db > 0 \,,\label{E: cl2}
      \end{eqnarray}
            where the second to last inequality follows from
            Assumption~\ref{A: LMLB2} that the full
            mixed partial derivative of $f$ exists
            and is continuous at $\m{x}_0$ from
            which we get, by definition of continuity,
            that there exists a large enough
            positive integer $n_1 := n_1(d,\m{x}_0,\m{h})
            \geq n_0(3d,\m{x}_0,\m{h})$ such that
            \[
                (-1)^d \left. \frac{\partial^d f}
                    {\partial x_1 \cdots \partial
                    x_d}(\m{x})\right|_{\m{x}=\m{y}}
                    - (-1)^db \geq
                    - 2^{-1}(1-\theta)(-1)^db
            \]
            holds true for all $\m{y}\in I_n(3d)$ and $n\geq n_1$.
            The result in \eqref{E: cl2} suggests that
            \[
                (-1)^d V_{f_n}[\m{x},\m{y}] \equiv (-1)^d
                \int_{(\m{x},\m{y}]}\left\{\left. \frac{\partial^d f_n}{\partial
                w_1 \cdots \partial w_n}
                (\m{w})\right|_{\m{w}=\m{u}}\right\}\,\ud\m{u} \geq 0
            \]
            holds true for all $d$-boxes $(\m{x},\m{y}]$ with
            $\m{x},\m{y}\in I_n(3d)$ and $n\geq n_1$.

            The last case not considered is the one that exactly
            one between $\m{x}$ and $\m{y}$, in the $d$-box
            $[\m{x},\m{y}]$, is an element of $I_n(3d)$.
            See also Figure~\ref{F: Song Perturbation}. For this case,
            we can appeal to Lemma~\ref{L: differencing} by setting
            $[\m{x}_0,\m{y}_0] := [\m{x},\m{y}]\cap I_n(3d)$ -- the latter
            being well-defined as the intersection of two rectangles is
            itself an rectangle. Then, from Lemma~\ref{L: differencing} and
            \eqref{E: cl1}, we have,
            \[
                (-1)^d V_{f_n}[\m{x},\m{y}] = (-1)^d
                V_{f_n}[\m{x}_0,\m{y}_0] + (-1)^d \sum_{i=1}^{m}\left\{
                V_{f_n}[\m{x}_i,\m{y}_i]\right\} \geq 0 + 0 = 0\,,
            \]
            exactly since $[\m{x}_i,\m{y}_i]\subseteq I_n^{c}(3d)$ for
            all $i \in \{1,2,\dots,m\}$ (where $m$ is as defined in
            Lemma~\ref{L: differencing}). For completeness, notice that we
            were not concerned above with end-point discontinuities of $f$
            (or $f_n$) on the entailed rectangle, subsets of
            $I_n(3d)$, as, in fact, $f$ (and $f_n$) is (are)
            continuous there for $n\geq n_1$, by Assumption~\ref{A: LMLB2}.

            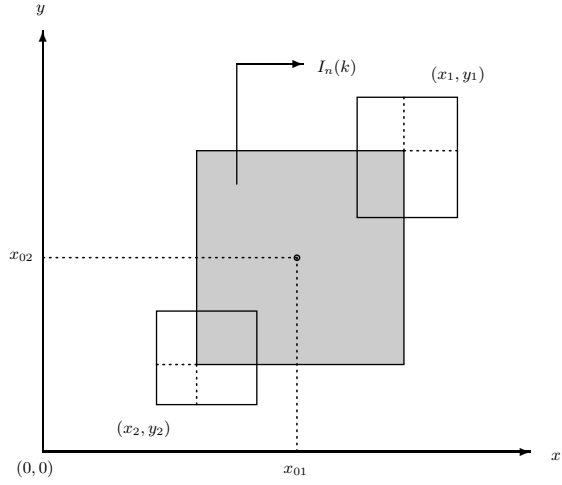
\begin{figure}[htp]
                \centering
                \scalebox{0.7}{
                    \setlength{\unitlength}{0.254mm}
                    \begin{picture}(422,357)(95,-536)
                            \allinethickness{0.254mm}\put(120,-520){\vector(0,1){315}} 
                            \allinethickness{0.254mm}\put(120,-520){\vector(1,0){365}} 
                            \allinethickness{0.254mm}\special{sh 0.2}\path(235,-295)(390,-295)(390,-455)(235,-455)(235,-295) 
                            \allinethickness{0.254mm}\special{sh 0.3}\put(310,-375){\ellipse{4}{4}} 
                            \allinethickness{0.254mm}\dottedline{5}(310,-375)(310,-520) 
                            \allinethickness{0.254mm}\dottedline{5}(310,-375)(120,-375) 
                            \put(95,-376){\shortstack{$x_{02}$}} 
                            \put(300,-536){\shortstack{$x_{01}$}} 
                            \put(325,-236){\shortstack{$I_n(k)$}} 
                            \allinethickness{0.254mm}\path(265,-230)(265,-320) 
                            \put(500,-526){\shortstack{$x$}} 
                            \put(115,-191){\shortstack{$y$}} 
                            \put(100,-536){\shortstack{$(0,0)$}} 
                            \allinethickness{0.254mm}\put(265,-230){\vector(1,0){50}} 
                            \allinethickness{0.254mm}\path(355,-255)(430,-255)(430,-345)(355,-345)(355,-255) 
                            \allinethickness{0.254mm}\dottedline{5}(390,-295)(390,-255) 
                            \allinethickness{0.254mm}\dottedline{5}(390,-295)(430,-295) 
                            \put(410,-241){\shortstack{$(x_1, y_1)$}} 
                            \allinethickness{0.254mm}\path(205,-415)(280,-415)(280,-485)(205,-485)(205,-415) 
                            \allinethickness{0.254mm}\dottedline{5}(235,-455)(205,-455) 
                            \allinethickness{0.254mm}\dottedline{5}(235,-455)(235,-485) 
                            \put(175,-506){\shortstack{$(x_2,y_2)$}} 
                    \end{picture}
                }
                \caption[Closure under perturbation in $\ffs$]{
                Perturbation rectangle $I_n(k)$, for the case $d=2$,
                with two rectangles intersecting $I_n(k)$ but otherwise not subsets of it.}
                \label{F: Song Perturbation}
            \end{figure}

            All these observations finally yield that $(-1)^d
            V_{f_n}[\m{x},\m{y}]\geq 0$ holds true for all
            $d$-boxes $[\m{x},\m{y}]$ and thus
            Theorem~\ref{T:InversionFormulaSMU}
            asserts that $f_n \in \ffs$ for all $n\geq n_1$.
            \qedhere
        \end{proof}

        We are ready to prove the main proposition of this section.

        \begin{proof}
                Recall Proposition~\ref{P: closureSMU}.
                First, we establish that
                \begin{equation}\label{E: MML1}
                    \int_{I_n} g_n(\m{x})\,\ud\m{x} = (-1)^d b
                    \prod_{i=1}^{d}\left\{h_i^2\right\}\cdot
                    n^{-\frac{2}{3}}\,,
                \end{equation}
                where, hereafter, $I_n$ will be the short-hand form for $I_n(3d)$.         By definition, notice that,
\begin{eqnarray*}
\lefteqn{\frac{1}{b}\int_{I_n} g_n(\m{x})\,\ud\m{x}
                      = \int_{I_n}\int_{I_n}
                    \prod_{i=1}^{d}\left\{\mathbbm{1}_{[x_i\leq u_i]}\right\}
                    h_n(\m{u})\,\ud\m{u}\ud\m{x} } \\
&=& \int_{I_n} h_n(\m{u})\left\{\int_{I_n}
                    \mathbbm{1}_{(\m{0},\m{u}]}(\m{x})\,\ud\m{x}\right\}\,\ud\m{u}\\
& = & \int_{I_n} \prod_{i=1}^{d}\left\{ u_i -
                    \left(\xx\right)\right\} h_n(\m{u})\,\ud\m{u} \\
& =& \prod_{i=i}^{d}\Biggl\{\int_{\xx}^{\XX}\Bigl(
                    [u_i -(\xx)]\times\\
&& \ \ \ \times [\mathbbm{1}_{[\xx,x_{0i}]}(u_i)
                    - \mathbbm{1}_{(x_{0i},\XX]}(u_i)]\Bigr)\,\ud
                    u_i\Biggr\} \\
& =& \prod_{i=1}^{d}\Biggl\{\int_{\xx}^{x_{0i}}
                    [u_i-(\xx)]\,\ud u_i + \\
&& \ \ \ -  \int_{x_{0i}}^{\XX}[u_i-(\xx)]\,\ud
                    u_i\Biggr\} \\
& =& \prod_{i=1}^{d}\Biggl\{ \int_{0}^{h_i
                    n^{-\frac{1}{3d}}}[-y +h_i
                    n^{-\frac{1}{3d}}]\,\ud y
          - \int_{0}^{h_i n^{-\frac{1}{3d}}}[w+h_i
                    n^{-\frac{1}{3d}}]\,\ud w\Biggr\} \\
& =& \prod_{i=1}^{d}\left\{\int_{0}^{h_i
                    n^{-\frac{1}{3d}}} \left(-2y\right)\,\ud y \right\}
          = (-1)^d \prod_{i=1}^{d}\left\{h_i^2
                    n^{-\frac{2}{3d}}\right\} = (-1)^d
                    \prod_{i=1}^{d}\left\{h_i^2\right\}\cdot
                    n^{-\frac{2}{3}}\,,
\end{eqnarray*}
thus yielding \eqref{E: MML1}.

                We next derive another equality, the most important fact
                about it being the factor $n^{-1}$ on the right hand side:
                \begin{equation}\label{E: MML2}
                    \int_{I_n} g^2_n(\m{x})\,\ud\m{x} =
                    {\left(\frac{8}{3}\right)}^{d} b^2
                    \prod_{i=1}^{d}\left\{ h^3_i\right\}\cdot n^{-1}\,.
                \end{equation}
                Before we start deriving \eqref{E: MML2}, let us first define four rectangles $R^i_j$ with $j=1,2,3,4$ for each $i\in\{1,2,\dots,d\}$:
                \begin{enumeratei}
                    \item $R^i_1 = \left[\xx,x_{0i}\right]\times\left[\xx,x_{0i}\right],$
                    \item $R^i_2 = \left[\xx,x_{0i}\right]\times\left(x_{0i},\XX\right],$
                    \item $R^i_3 = \left(x_{0i},\XX\right]\times\left[\xx,x_{0i}\right],$
                    \item $R^i_4 = \left(x_{0i},\XX\right]\times\left(x_{0i},\XX\right].$
                \end{enumeratei}

                Then, by definition:
                \begin{eqnarray}
                    \lefteqn{\frac{1}{b^2}\int_{I_n} g^2_n(\m{x})\,\ud\m{x}
                              = \int_{I_n}{\left\{ \int_{I_n} h_n(\m{u})
                              \mathbbm{1}_{[\m{x}\le \m{u}]}\,\ud\m{u}\right \}}^2\,\ud\m{x}} \nonumber \\
                     &=& \int_{I_n}\int_{I_n}\int_{I_n} h_n(\m{u})h_n(\m{v})
                              \mathbbm{1}_{[\m{x} \le \m{u}\wedge\m{v}]}\,\ud\m{v}\ud\m{u}\ud\m{x} \nonumber\\
                     &=& \int_{I_n}\int_{I_n} \biggl\{\prod_{i=1}^{d}\bigl[(u_i\wedge
                                v_i)-(\xx)\bigr]\times
                              h_n(\m{u})h_n(\m{v})\biggl\}\,\ud\m{v}\ud\m{u}  \nonumber\\
                    & = & \prod_{i=1}^{d} \biggl\{ \int_{R^i_1+R^i_3}
                               \bigl[(u\wedge v)-(\xx)\bigr]\,\ud v\ud u + \nonumber \\
                    & &\ \ -\ 2\int_{R^i_2} \bigl[(u\wedge v) -(\xx)\bigr]\,\ud v \ud u \biggr\}\nonumber \\
                    & = & 2^d \prod_{i=1}^{d} \left\{ S_{1i} + S_{2i} -
                                S_{3i} \right\}\,, \label{E: MML3}
                \end{eqnarray}
                where the last equality follows by symmetry and
                Fubini-Tonelli and the integrals in the braces are to be
                evaluated below:
                \begin{align*}
                    S_{1i}&\equiv \int\limits_{\xx}^{x_{0i}}\int\limits_{v}^{x_{0i}}
                    \left\{ v - \left(\xx\right)\right\}\,\ud u\ud v \\
                    &= \int\limits_{\xx}^{x_{0i}}\left\{(x_{0i}-v)\left(v-x_{0i}+h_i
                    n^{-\frac{1}{3d}}\right)\right\}\,\ud v \\
                    &= \int\limits_{\XX}^{h_i n^{-\frac{1}{3d}}} \left\{y \left(-y
                    +h_i n^{-\frac{1}{3d}}\right)\right\}\,\ud
                    y\quad\text{[change of variable]}\\
                \intertext{while, again, by a change of variable argument:}
                    S_{2i} &\equiv \int\limits_{x_{0i}}^{\XX}\int\limits_{v}^{\XX}
                    \left\{ v - \left(\xx\right)\right\}\,\ud u\ud v \\
                    &= \int\limits_{x_{0i}}^{\XX} \left\{\left[(x_{0i}-v)+h_i
                    n^{-\frac{1}{3d}}\right]\left[(v-x_{0i})+h_i
                    n^{-\frac{1}{3d}}\right]\right\}\,\ud v \\
                    &= \int\limits_{0}^{h_i n^{-\frac{1}{3d}}}\left\{\left(-y +h_i
                    n^{-\frac{1}{3d}}\right)\left(y + h_i
                    n^{-\frac{1}{3d}}\right)\right\}\,\ud y\;,\ \\
                \intertext{and similarly:}
                    S_{3i} &\equiv \int_{\xx}^{x_{0i}}\left\{h_i
                    n^{-\frac{1}{3d}} \left(v - x_{0i} +h_i
                    n^{-\frac{1}{3d}}\right)\right\}\,\ud v \\
                    &= h_i n^{-\frac{1}{3d}} \int_{0}^{h_i
                    n^{-\frac{1}{3d}}} \left\{ h_i n^{-\frac{1}{3d}} - y
                    \right\}\,\ud y\;.
                \end{align*}

                Let now $q_i := h_i n^{-{1}/{3d}}$, for
                $i\in\{1,2,\dots,d\}$, and observe that
                \begin{eqnarray*}
                    S_{1i}+S_{2i}-S_{3i} = \int_{0}^{q_i} \left\{ y(q_i -
                    y) + q^2_i - y^2 + q^2_i - q_i y\right\}\,\ud y
                    = \cdots = \frac{4}{3} h^3_i n^{-\frac{1}{d}}\,,
                \end{eqnarray*}
                so that plugging all these in~\eqref{E: MML3} yields the
                desired~\eqref{E: MML2}.

                Now, recall from the definition of $f_n$ that $\theta\in
                (0,1)$ was arbitrary but fixed. Also, from $\int_{\s}
                f_n(\m{x})\,\ud\m{x} = 1$ we can get an explicit expression for the normalizing constant $d_n$:
                \begin{eqnarray}
                    d_n &=& \int_{I_n} f(\m{x})\,\ud\m{x} + \int_{I_n^c}
                    f(\m{x})\,\ud\m{x} + \theta \int_{I_n}
                    g_n(\m{x})\,\ud\m{x} \notag \\
                    &=& 1 + \theta \int_{I_n} g_n(\m{x})\,\ud\m{x}
                    = 1 + (-1)^d \theta b
                    \prod_{i=1}^{d}\left\{h^2_i\right\}\cdot
                    n^{-\frac{2}{3}}\,, \label{E: MML4}
                \end{eqnarray}
                where the second to last equality follows from $\int_{\s}
                f(\m{x})\,\ud\m{x} = 1$, while the last equality follows
                from~\eqref{E: MML1}. Notice from~\eqref{E: MML4} that $d_n
                \downarrow 1$ as $n \uparrow\infty.$ Also, from the easily
                verifiable identity $g_n(\m{x}_0)=(-1)^d b
                \prod_{i=1}^{d}\left\{h_i\right\} n^{-{1}/{3}}$, we
                have
                \begin{eqnarray}
                    n^{\frac{1}{3}}\left|f_n(\m{x}_0)-f(\m{x}_0)\right|
                    &=&   n^{\frac{1}{3}}\left|\,\frac{f(\m{x}_0) + (-1)^d b
                    \prod_{i=1}^{d}\left\{h_i\right\} n^{-\frac{1}{3}}}{d_n}
                    - f(\m{x}_0)\right| \notag \\
                    &=& \left|\, n^{\frac{1}{3}}\left\{\frac{1}{d_n}-1\right\}
                    f(\m{x}_0) + \frac{(-1)^d b \theta
                    \prod_{i=1}^{d}\left\{h_i\right\}}{d_n}\right| \notag \\
                    &\longrightarrow& (-1)^d b \theta
                    \prod_{i=1}^{d}\left\{h_i\right\}\;(>0)\,,\quad\text{as $n\to\infty$}.
                    \label{E: MML5}
                \end{eqnarray}
                Also,
                \begin{eqnarray}
                    2n h^2(f_n,f)
                    & = &  n\int_{I_n}{\left\{\sqrt{f_n(\m{x})}-\sqrt{f(\m{x})}\right\}}^2\,\ud\m{x}
                              +n\int_{I_n^{c}}{\left\{\sqrt{f_n(\m{x})}-\sqrt{f(\m{x})}\right\}}^2\,\ud\m{x}
                                \notag \\
                    & = &  n \int_{I_n} {\left\{\frac{f_n(\m{x})-f(\m{x})}
                               {\sqrt{f_n(\m{x})}+\sqrt{f(\m{x})}}\right\}}^2\,\ud\m{x}
                               + \delta^2_n\int_{I_n^{c}} f(\m{x})\,\ud\m{x}\,, \label{E: MML6}
                \end{eqnarray}
                where,
                \begin{eqnarray*}
                    \delta_n
                    &\equiv& \sqrt{n}\left\{1-\frac{1}{\sqrt{d_n}}\right\}
                                   = \sqrt{n} \left\{ \frac{\sqrt{d_n}-1}{\sqrt{d_n}} \right\}\\
                    &=& \frac{ \sqrt{n}\left \{\sqrt{1 + \bo\left(n^{-\frac{2}{3}} \right)} - 1\right \}}
                                  {\sqrt{d_n}}
                    \rightarrow 0 \,,\quad\text{as $n\to\infty$},
                \end{eqnarray*}
                with the convergence on the last display following
                from~\eqref{E: MML4}. Applying this to \eqref{E: MML6}, we
                have:
                \begin{equation}
                    2n h^2(f_n,f) = n \int_{I_n} {\left\{\frac{f_n(\m{x})-f(\m{x})}
                    {\sqrt{f_n(\m{x})}+\sqrt{f(\m{x})}}\right\}}^2\,\ud\m{x}
                    + \lo(1) \label{E: MML7}
                \end{equation}
                as $n\to\infty$, because
                $0\leq \int_{I_n^{c}} f(\m{x})\,\ud\m{x} \leq 1$.

                For fixed $n\in\NN$, such that $f$ and $g_n$ be continuous
                and strictly positive on
                $I_n$, let $\m{x}_{(n)}$ and $\m{x}^{(n)}$ denote,
                respectively, a minimizer and a maximizer of $f$ on
                the compact set $I_n$. Let also
                $\m{y}_{(n)}$ and $\m{y}^{(n)}$ denote,
                respectively, a minimizer and a maximizer of $g_n$ on
                the compact set $I_n$. Observe that, since $I_n$ is a
                decreasing sequence of compact sets converging to
                $\{\m{x}_0\}$, all of $\m{x}_{(n)}$, $\m{x}^{(n)}$,
                $\m{y}_{(n)}$ and $\m{y}^{(n)}$ converge to $\m{x}_0$
                as $n\to\infty$. Also,
                \begin{eqnarray}
                 \sup_{\m{x}\in I_n} \left|
                    \frac{f_n(\m{x}) - f(\m{x})}{f(\m{x})}\right|
                  &= & \sup_{\m{x}\in I_n}
                            \left| \left(\frac{1}{d_n} - 1\right) +
                             \frac{\theta g_n(\m{x})}
                             {d_n f(\m{x})}\right|  \nonumber \\
                  & \leq & \left( 1 - \frac{1}{d_n}\right)
                             + \frac{\theta \sup_{\m{x}\in I_n}
                                 \left\{ g_n(\m{x})\right\}}
                             {d_n \inf_{\m{x}\in I_n}\left\{
                             f(\m{x})\right\}}\nonumber \\
                    &\rightarrow & 0\,,\quad\text{as $n\to\infty$\,,}
                    \label{E: Jongbloed1}
                \end{eqnarray}
                because $g_n$ is pointwise non-increasing in $n\in\NN$, $g_n(\m{x}_0) =
                \bo\left(n^{-1/3}\right)$ and $f(\m{x}_0) > 0$.

                Also,
                \begin{eqnarray*}
                    D_1(n) &\equiv& \int_{I_n} {\left\{ f_n(\m{x}) -
                    f(\m{x}) \right\}}^2\,\ud\m{x} \\
                    &=& \frac{1}{d^2_n}\int_{I_n}\left\{ \theta^2
                    g^2_n(\m{x}) - \bo\left(n^{-\frac{2}{3}}\right) f(\m{x})
                    g_n(\m{x}) + \bo\left(n^{-\frac{4}{3}}\right)
                    f^2(\m{x})\right\}\,\ud\m{x}
                \end{eqnarray*}
                and noticing that
                \begin{equation*}
                    0\leq \int_{I_n}\left\{ g_n(\m{x}) f(\m{x})\right\}\,\ud\m{x}
                    \leq f\left(\m{x}^{(n)}\right)
                    \int_{I_n}\left\{g_n(\m{x})\right\}\,\ud\m{x}
                    = \bo\left(n^{-\frac{2}{3}}\right)\,,
                \end{equation*}
                so that,
                \begin{eqnarray}
                    nD_1(n) &=&
                    \frac{n}{d^2_n}\left\{{\left(\frac{8}{3}\right)}^{d}
                    \theta^2 b^2 \prod_{i=1}^{d}\left\{h_i^3\right\}\cdot
                    n^{-1} + \lo\left(n^{-\frac{4}{3}}\right)\right\} \notag \\
                    &\longrightarrow&
                    {\left(\frac{8}{3}\right)}^{d} \theta^2 b^2
                    \prod_{i=1}^{d}\left\{h_i^3\right\}\,,\quad\text{as $n\to\infty$}\,. \label{E: MML10}
                \end{eqnarray}

                Now, since $f$ is block-decreasing, we have,
                \[
                    0 < f\left(\m{x}_0 +
                    n^{-\frac{1}{3d}}\m{I}_{d}\m{h}\right)
                    \leq f(\m{x})\leq f\left(\m{x}_0 -
                    n^{-\frac{1}{3d}}\m{I}_{d}\m{h}\right)
                \]
                for all $\m{x}\in I_n$ and $n \geq n_1.$ Hence,
                \[
                    \frac{nD_1(n)}{f\left(\m{x}_0 -
                    n^{-\frac{1}{3d}}\m{I}_{d}\m{h}\right)}
                    \leq n \int_{I_n}
                    \frac{{\left\{f_n(\m{x}) -
                    f(\m{x})\right\}}^2}{f(\m{x})}\,\ud\m{x}
                    \leq \frac{nD_1(n)}{f\left(\m{x}_0 +
                    n^{-\frac{1}{3d}}\m{I}_{d}\m{h}\right)}
                \]
                which, ahead with \eqref{E: MML10} and
                sandwich, yields
                \[
                    n \int_{I_n}
                    \frac{{\left\{f_n(\m{x}) -
                    f(\m{x})\right\}}^2}{f(\m{x})}\,\ud\m{x}
                    \longrightarrow
                    {\left(\frac{8}{3}\right)}^{d} \theta^2 b^2 \cdot
                    \frac{\prod_{i=1}^{d}\left\{h^3_i\right\}}{f(\m{x}_0)}\,,\quad\text{as
                    $n\to\infty$}.
                \]
                Applying all of the above to \eqref{E: MML7}, and
                appealing to Lemma~2 of \citet{MR1792307}, we get
                \begin{eqnarray}
                    n h^2(f_n,f)
                    & = & \frac{1}{8}\int_{I_n}
                               \frac{{\left\{f_n(\m{x}) -
                               f(\m{x})\right\}}^2}{f(\m{x})}\,\ud\m{x}
                                +\lo(1)  \\
                    & \rightarrow &
                    \frac{8^{d-1}}{3^d f(\m{x}_0)} \theta^2 b^2
                    \prod_{i=1}^{d}\left\{h^3_i\right\}
                    \label{E: MML11}
                \end{eqnarray}
              as $n\to\infty$,  so that by applying \eqref{E: MML5} and \eqref{E: MML11} to
                Lemma~\ref{P: GJ}, we get
                \begin{eqnarray*}
                    \lefteqn{ \varliminf_{n\to\infty}\inf_{T_n}\max\left\{
                    \E_{f_n}\left\{n^{\frac{1}{3}}\left|T_n-
                    f_n(\m{x}_0)\right|\right\} ,
                    \E_{f}\left\{n^{\frac{1}{3}}\left|T_n -
                    f(\m{x}_0)\right|\right\}\right\} }\\
                    & \geq & \frac{1}{4}\left\{(-1)^d b\right\} \theta c
                    \exp\left\{-\frac{2^{3d-2}}{3^d f(\m{x}_0)} \theta^2 b^2
                    c^3 \right\}
                     =:  G_{f,\m{x}_0}(c,\theta)
                \end{eqnarray*}
where $c \equiv \prod_{i=1}^{d}\left\{h_i\right\}$.
                For a fixed $\theta \in (0,1)$ the maximum of
                $G_{f,\m{x}_0}(c,\theta)$ is attained at
                \[
                    c(\theta)={\left\{\frac{3^{d-1} f(\m{x}_0)}{2^{3d-2}
                    \theta^2 b^2}\right\}}^{\frac{1}{3}}
                \]
                and is equal to
                \[
                    G_{f}\left(c(\theta),\theta\right) =
                    \frac{e^{-\frac{1}{3}}}{2^{d}}{\left\{3^{d-1}\theta
                    \right\}}^{\frac{1}{3}}{\left\{(-1)^{d}
                    \left.\frac{\partial^d f(\m{x})}{\partial x_1 \cdots \partial
                    x_d}\right|_{\m{x}=\m{x}_0}
                    f(\m{x}_0)\right\}}^{\frac{1}{3}},
                \]
                the latter being an increasing function of $\theta\in (0,1)$.

                This suggests that
                \begin{gather}
                    \varliminf_{n\to\infty}\inf_{T_n}\max\left\{
                    \E_{f_n}\left\{n^{\frac{1}{3}}\left|T_n-
                    f_n(\m{x}_0)\right|\right\} ,
                    \E_{f}\left\{n^{\frac{1}{3}}\left|T_n -
                    f(\m{x}_0)\right|\right\}\right\} \notag \\
                    \geq \frac{e^{-\frac{1}{3}}}{2^{d}}{\left\{
                    \theta\cdot3^{d-1}\right\}}^{\frac{1}{3}}{\left\{(-1)^d
                    \left.\frac{\partial^d f(\m{x})}{\partial x_1 \cdots \partial
                    x_d}\right|_{\m{x}=\m{x}_0} \cdot f(\m{x}_0)
                    \right\}}^{\frac{1}{3}}.\notag
                \end{gather}
                Overall, we are allowed to take $\theta \uparrow 1$ in the above
                display, even if $\theta = 1$ is not a valid configuration, yielding the
                lower bound in the wording of the proposition. The proof is
                thus complete. \qedhere
        \end{proof}

\section{Discussion and open problems} \label{S: discussion}
Once consistency has been established, interest focuses on rates of
convergence of the MLE
    and other properties, including the behavior of $\widehat{f}_n$ at zero and pointwise limiting
    distributions.  We have the following conjectures concerning the MLE $\widehat{f}_n$ for the
    class $\ffs(d)$.  Work is currently underway on all of these further problems.  \\
    \smallskip

\par\noindent
{\bf Conjecture 1.}    If $f_0 (0) < \infty$, then we conjecture that\\
$P_0 ( \widehat{f}_n (0) \le M (\log n)^{d-1} ) \rightarrow 1$
    for some $M>0$.  \\
    \smallskip

\par\noindent
{\bf Conjecture 2.}
      If $f_0 (0) < \infty$ and $f_0 $ is concentrated on $[\m{0}, M \m{1} ]$ for some $0 < M < \infty$, then
    $h ( \widehat{f}_n , f_0 ) = O_p ( n^{-1/3} ( \log n )^{\gamma} )$ for some $\gamma$ depending only on $d$.\\
\smallskip

    Concerning rates of convergence of the estimators at a fixed point, we do not yet have any upper bound results
    to accompany the lower bound results of Proposition~\ref{P: LAMLB2}.
    Thus there remain the following two possibilities:
    (a) the pointwise rate of convergence under Assumption~\ref{A: LMLB2}
    is $n^{1/3}$, and we expect convergence in distribution with
    the rate $n^{1/3}$; or, (b) the lower bound given in Proposition~\ref{P: LAMLB2}
    is not yet sharp, and we should expect log terms in
    the rate (as might be expected from the covering number results of \citet{MR2341952}).
    Our corresponding conjectures for these two possible scenarios
    are given below as Conjectures 3a and 3b respectively.

\par\noindent
{\bf Conjecture 3a.}
 Suppose that $f_0$ has $\partial^d f_0 ( \m{x} ) / \partial x_1  \cdots \partial x_d $ continuous in a neighborhood of
    $\m{x}_0$ with
    $$
    \partial^d f_0 (\m{x}_0 ) \equiv  \frac{\partial^d  f_0 ( \m{x} )}{ \partial x_1  \cdots \partial x_d } \bigg |_{\m{x} = \m{x}_0} \not= 0.
$$
    Let $\{ W (\m{t} ) :  \ \m{t} \in \RR^d \}$ be a  $2^d$-sided Brownian sheet process on $\RR^d$ and let
    $$
    \YY(\m{t} ) \equiv \sqrt{f_0 (\m{x}_0) } W(\m{t} ) + \frac{(-1)^d}{2^d}  (-1)^d \partial^d f_0 (\m{x}_0)  | \m{t} | ^2 .
    $$
       Then, in keeping with our lower bound results of Section~4, we conjecture that
    $$
    n^{1/3} ( \widehat{f}_n (\m{x}_0 ) - f_0 (\m{x}_0) ) \rightarrow_d  \partial^d \HH (\m{t}) |_{\m{t} = \m{0} }
    $$
    where the process $\HH$ is determined by
    \begin{eqnarray*}
    &&(i) \ \ \  \HH ( \m{t} ) \ge \YY (\m{t} ) \qquad \mbox{for all} \ \ \m{t} \in \RR^d , \\
    &&(ii) \  \int_{\RR^d} ( \HH (\m{t}) - \YY (\m{t} ) ) d (\partial^d \HH (\m{t} )) = 0,  \ \ \mbox{and} \\
    && (iii) \ \, V_{\partial^d \HH} [ \m{u} , \m{v} ) \ge 0  \qquad \mbox{for all} \ \ \m{u} \le  \m{v} \in \RR^d .
    \end{eqnarray*}
    Partial results concerning Conjecture 3a were obtained in \citet{Pavlides:08}.

    \par\noindent
   {\bf Conjecture 3b.}  As suggested in part by the covering number results of \citet*{MR2341952}, the
   pointwise rate of convergence is $(n/(\log n)^{d-1/2})^{1/3}$.    This would entail an improved version
   of Proposition~\ref{P: LAMLB2}.  In this case we do not yet have conjectures concerning the limiting distribution.

\bigskip

\par\noindent
{\bf Acknowledgments:}  We owe thanks to Marina Meila, Fritz Scholz,
and Arseni Seregin for helpful discussions concerning the proof of
uniqueness, and especially Lemmas 3.3 and 3.4.

\bibliographystyle{ims}
\bibliography{combined_g}

\end{document}